\documentclass[12pt,a4paper,oneside,reqno]{article}

\usepackage[utf8]{inputenc}
\usepackage[english]{babel}
\usepackage[left=1in,right=1in,top=1in,bottom=1in]{geometry}
\usepackage{amsfonts,amssymb,amsmath,amsthm}
\usepackage[symbol]{footmisc}
\usepackage{longtable,array}
\newcolumntype{L}[1]{>{\raggedright\arraybackslash$}p{#1}<{$}}
\usepackage{enumerate}
\usepackage{cite}
\usepackage{xcolor}
\newcommand{\C}{\mathbb{C}}
\usepackage{hyperref}
\hypersetup{hidelinks}
\usepackage{fancyhdr}
\theoremstyle{plain}
\newtheorem{theorem}{Theorem}
\newtheorem{lemma}[theorem]{Lemma}

\newtheorem{corollary}[theorem]{Corollary}
\newtheorem{conjecture}{Conjecture}
\theoremstyle{definition}
\newtheorem{remark}[theorem]{Remark}
\newtheorem{definition}[theorem]{Definition}
\makeatletter
\@removefromreset{theorem}{section}\@removefromreset{equation}{section}

\makeatother
\usepackage{adjustbox}
\usepackage{tikz}\usetikzlibrary{arrows}

\IfFileExists{fouriernc.sty}{\usepackage{fouriernc}}{}

\newcommand{\doi}[1]{\href{https://doi.org/#1}{doi:#1}}
\newcommand{\arxiv}[1]{\href{https://arxiv.org/abs/#1}{arXiv:#1}}
\newcommand{\arxivsearch}[1]{\href{https://arxiv.org/search/?query=#1}{arXiv}}

\newcommand\Tstrut{\rule{0pt}{2.6ex}}\newcommand\Bstrut{\rule[-1.1ex]{0pt}{0pt}}

\begin{document}

\noindent\fbox{\parbox{\dimexpr\textwidth-2\fboxsep-2\fboxrule\relax}{\small
\textbf{Corrected version.} This is a corrected version of the paper A. J. Calder\'on, A. Fern\'andez Ouaridi, I. Kaygorodov,
\emph{The classification of 2-dimensional rigid algebras},
Linear and Multilinear Algebra, 68 (2020), no.~4, 828--844,
\doi{10.1080/03081087.2018.1519009}.}}

\bigskip

\bigskip

\begin{center}
{\Large The classification of $2$-dimensional rigid algebras}

\bigskip
\bigskip

{\bf Antonio Jes\'us Calder\'on, Amir Fern\'andez Ouaridi and Ivan Kaygorodov}
\end{center}

\bigskip

\bigskip
\noindent{\small\textbf{Abstract.} Using the algebraic classification of all $2$-dimensional algebras,
we give the algebraic classification of all $2$-dimensional rigid, conservative and terminal algebras
over an algebraically closed field of characteristic 0.
We have the geometric classification of the variety of $2$-dimensional terminal algebras,
and based on the geometric classification of these algebras we formulate some open problems.}

\bigskip


\bigskip


\bigskip

\tableofcontents

\bigskip

\newpage
\section*{Introduction}
\addcontentsline{toc}{section}{Introduction}
\subsection*{Conservative, terminal and rigid algebras}

 In 1972, Kantor introduced the notion a conservative algebra as a generalization of Jordan algebras \cite{Kantor72}. Unlike other classes of non-associative algebras, this class is not defined by a set of identities. Instead they are defined in the following way. 
 
Consider an algebra as a vector space $\mathbb V$ over a field {\bf k}, together with an
 element $\mu$ of $Hom(\mathbb V \otimes \mathbb V,  \mathbb V),$ so that $a \cdot b =\mu(a \otimes b).$ 
Given a linear map 
$\mathcal A:\mathbb V\rightarrow \mathbb V$ and a bilinear map
$\mathcal B:\mathbb V\times \mathbb V\to \mathbb V$, we define the product of a linear map and a bilinear map as the map $[\mathcal A,\mathcal B]:\mathbb V\times \mathbb V\to \mathbb V$ such that
$$[\mathcal A,\mathcal B](x,y)=\mathcal A(\mathcal B(x,y))-\mathcal B(\mathcal A(x),y)-\mathcal B(x,\mathcal A(y)), \mbox{ for all }x,y\in \mathbb V.$$
For an algebra $(\mathbb V, \mathcal P)$ with a multiplication $\mathcal P$ and $x\in \mathbb V$ we denote by $L_x^{\mathcal P}$ the operator of left multiplication by $x$. 
Thus, Kantor defines conservative algebras as follow.

\begin{definition}
An algebra $(\mathbb V, \mathcal P)$, where $\mathbb V$ is the vector space and $\mathcal P$ is the multiplication, is called a conservative algebra if there is a new multiplication $\mathcal F: \mathbb V\times \mathbb V\rightarrow \mathbb V$ such that 
\begin{equation}\label{uno}
[L_b^{\mathcal P},[ L_a^{\mathcal P}, {\mathcal P}]]=-[L_{{\mathcal F}(a,b)}^{\mathcal P},{\mathcal P}] 
\textrm{, for all $a, b \in \mathbb V$.}
\end{equation}
Simple calculations take us to the following identity with an additional multiplication $\mathcal F$, which must  hold for all $a, b, x, y\in \mathbb V$:
\begin{equation}\label{dos}
\begin{split}
b(a(xy)-(ax)y-x(ay)) - a((bx)y) + (a(bx))y+(bx)(ay) \\ -a(x(by))+(ax)(by)+x(a(by))= \\
= - \mathcal F(a,b)(xy)+(\mathcal F(a,b)x)y + x(\mathcal F(a,b)y).
\end{split}
\end{equation}
\end{definition}

The class of conservative algebras is very vast \cite{KLP}. It includes:
 all associative algebras, 
 all quasi-associative algebras,
 all Jordan algebras, 
 all Lie algebras, 
 all (left) Leibniz algebras,
 all (left) Zinbiel algebras, 
 all terminal algebras
 and many other classes of algebras.

There are some properties of conservative algebras which are similar to wonderful properties of Lie algebras.
The conservative algebra $W(n)$ plays a similar role in the theory of conservative algebras as the Lie algebra of all $n\times n$ matrices $gl_n$ plays in the theory of Lie algebras. 
Namely, in~\cite{Kantor90}  Kantor considered the category $\mathcal{S}_n$ whose objects are conservative algebras of non-Jacobi dimension $n,$ and proven that the algebra $W(n)$ is the universal attracting object in this category, i.e., for every algebra $M$ of $\mathcal{S}_n$ there exists a canonical homomorphism from $M$ into the algebra $W(n)$. In particular, all Jordan algebras of dimension $n$ with unity are contained in the algebra $W(n)$. 
Some properties of the product in the algebra $W(n)$ were studied in \cite{KLP,KV15,Kg17}.
 
In 1989, Kantor introduced the class of terminal algebras which is  subclass of the class of conservative algebras \cite{Kantor89}. To introduce the notion of terminal algebra, we first define the product of two bilinear maps. 
Given two bilinear maps 
$\mathcal A:\mathbb V\times\mathbb V\rightarrow \mathbb V$ and 
$\mathcal B:\mathbb V\times \mathbb V\to \mathbb V$, we define the operation $[\mathcal A,\mathcal B]:\mathbb V\times\mathbb V\times \mathbb V\to \mathbb V$ such that for all $x,y, z\in \mathbb V$:
\begin{multline*}[\mathcal A,\mathcal B](x,y,z)=\mathcal A(\mathcal B(x,y),z)+\mathcal A(x, \mathcal B(y,z))+\mathcal A(y,\mathcal B(x,z))\\-\mathcal B(\mathcal A(x,y),z)-\mathcal B(x,\mathcal A(y,z))-\mathcal B(y,\mathcal A(x,z)).\end{multline*}

Also, for $a\in \mathbb V$ and a bilinear map $\mathcal A:\mathbb V\times\mathbb V\rightarrow \mathbb V$, we introduce the operation $[\mathcal A,a](x)=\mathcal A(a,x)$. Thus, we define:

\begin{definition}
An algebra $(\mathbb V, \mathcal P)$, where $\mathbb V$ is a  vector space and $\mathcal P$ is a multiplication, is called a terminal algebra if it satisfies, for any $a\in \mathbb V$:
\begin{equation}\label{tress}
[[[{\mathcal P},a],{\mathcal P}],{\mathcal P}]=0.
\end{equation}

\end{definition}

The following remark is obtained by straightforward calculations.

\begin{remark}
Any commutative algebra satisfying (\ref{tress}) is a Jordan algebra. 
\end{remark}

The class of terminal algebras includes all Jordan algebras, 
all Lie algebras,
all (left) Leibniz algebras and some other types of algebras.

The following characterization of terminal algebras, proved by Kantor \cite[Theorem 2]{Kantor89}, provides a description of this class as a subclass of the class of conservative algebras.

\begin{remark}
An algebra $(\mathbb V, \mathcal P)$ is terminal if and only if, for any  $a, b \in \mathbb V$:
\begin{equation}[L_b^{\mathcal P},[ L_a^{\mathcal P}, {\mathcal P}]]=-[L_{2/3{\mathcal P}(a,b)+1/3{\mathcal P}(b,a)}^{\mathcal P},{\mathcal P}].
\end{equation}
\end{remark}

Then Kantor introduced a generalization of conservative algebras \cite{Kantor89trudy}:

\begin{definition}
An algebra $(\mathbb V, \mathcal P)$, where $\mathbb V$ is a vector space and $\mathcal P$ is a multiplication, is called a quasi-conservative algebra if there are a  multiplication $\mathcal F: \mathbb V\times \mathbb V\rightarrow \mathbb V$
and a  bilinear form $\phi: \mathbb V\times \mathbb V\rightarrow \bf k,$ such that 
\begin{equation}\label{quasi}
[L_b^{\mathcal P},[ L_a^{\mathcal P}, {\mathcal P}]]=-[L_{{\mathcal F}(a,b)}^{\mathcal P},{\mathcal P}] +\phi(a,b) \mathcal P
\textrm{, for all $a, b \in \mathbb V$.}
\end{equation}
\end{definition}

Let us recall that the structural Lie algebra $Str(\mathbb V,\mu)$ is the subalgebra of the Lie algebra $End( \mathbb V,  \mathbb V)$-generated by all operators of left multiplication $\mu a(b) = \mu (a\otimes  b),$ 
$a \in  \mathbb V,$ and denote by $R( \mathbb V,\mu)$ the minimal
submodule of the $Str( \mathbb V,\mu)$-module $Hom( \mathbb V \otimes  \mathbb V,  \mathbb V),$ containing $\mu.$ 
Following Kac and Cantarini \cite{ck10}, 
we can give the following
\begin{definition}
An algebra  $(\mathbb V, \mu)$  where $\mathbb V$ is a vector space and $\mu$ is a multiplication, is called a rigid algebra if it satisfies:
\begin{equation}\label{un} R( \mathbb V, \mu) = Str( \mathbb V, \mu)\mu +  {\bf k} \mu.
\end{equation}
\end{definition}

Thus, (\ref{un}) means a certain rigidity property. Namely, 
in the case of Jordan algebras, this property means that “small” deformations of
the product by the structural group produce an isomorphic algebra.
The class of rigid algebras is very vast:
it includes all associative algebras, all Jordan algebras, all Lie algebras, all conservative algebras and many other types of algebras.

\begin{remark}
An algebra $\mathbb A$ is rigid if and only if $\mathbb A$ is quasi-conservative.
\end{remark}

\subsection*{The classification of $2$-dimensional algebras}
The study of $2$-dimensional algebras has a very big history \cite{cfk18,GR11,kv17}.
To give the classification of $2$-dimensional algebras we have to introduce some notation used in the latest algebraic classification of $2$-dimensional algebras  \cite{kv17}.
Let us consider the action of the cyclic group $C_2=\langle \rho\mid \rho^2\rangle$ on ${\bf k}$ defined by the equality ${}^{\rho}\alpha=-\alpha$ for $\alpha\in{\bf k}$.
Now, fix some set of representatives of the orbits under this action and denote it by ${\bf k_{\geq 0}}$. For example, if ${\bf k}=\mathbb{C}$, then one can take $\mathbb{C}_{\geq 0}=\{\alpha\in\mathbb{C}\mid Re(\alpha)>0\}\cup\{\alpha\in\mathbb{C}\mid Re(\alpha)=0,Im(\alpha)\geq 0\}$.

Let us also consider the action of $C_2$ on ${\bf k}^2$ defined by the equality ${}^{\rho}(\alpha,\beta)=(1-\alpha+\beta,\beta)$ for $(\alpha,\beta)\in{\bf k}^2$.
Let us fix some set of representatives of the orbits under this action and denote it by $\mathcal{U}$. Let us also define $\mathfrak{T}=\{(\alpha,\beta)\in{\bf k}^2\mid \alpha+\beta=1\}$.

Given $(\alpha,\beta,\gamma,\delta)\in{\bf k}^4$, we define $\mathcal{D}(\alpha,\beta,\gamma,\delta)=(\alpha+\gamma)(\beta+\delta)-1$.
We define $\mathcal{C}_1(\alpha,\beta,\gamma,\delta)=(\beta,\delta)$, $\mathcal{C}_2(\alpha,\beta,\gamma,\delta)=(\gamma,\alpha)$, and $\mathcal{C}_3(\alpha,\beta,\gamma,\delta)=\left(\frac{\beta\gamma-(\alpha-1)(\delta-1)}{\mathcal{D}(\alpha,\beta,\gamma,\delta)},\frac{\alpha\delta-(\beta-1)(\gamma-1)}{\mathcal{D}(\alpha,\beta,\gamma,\delta)}\right)$ for $(\alpha,\beta,\gamma,\delta)$ such that $\mathcal{D}(\alpha,\beta,\gamma,\delta)\not=0$. Let us consider the set $\left\{\big(\mathcal{C}_1(\Gamma),\mathcal{C}_2(\Gamma),\mathcal{C}_3(\Gamma)\big)\mid \Gamma\in{\bf k}^4, \mathcal{D}(\Gamma)\not=0,\mathcal{C}_1(\Gamma),\mathcal{C}_2(\Gamma)\not\in \mathcal{T}\right\}\subset ({\bf k}^2)^3$.
One can show that the symmetric group $S_3$ acts on this set by the equality $${}^{\sigma}\big(\mathcal{C}_1(\Gamma),\mathcal{C}_2(\Gamma),\mathcal{C}_3(\Gamma)\big)=\big(\mathcal{C}_{\sigma^{-1}(1)}(\Gamma),\mathcal{C}_{\sigma^{-1}(2)}(\Gamma),\mathcal{C}_{\sigma^{-1}(3)}(\Gamma)\big) \mbox{ for }\sigma\in S_3.$$ Note that there exists a set of representatives of orbits under this action $\mathcal{\tilde V}$ such that if $(\mathcal{C}_1,\mathcal{C}_2,\mathcal{C}_3)\in \mathcal{\tilde V}$ and $\mathcal{C}_1\not=\mathcal{C}_2$, then $\mathcal{C}_3\not=\mathcal{C}_1,\mathcal{C}_2$. Let us fix such $\mathcal{\tilde V},$ and define
$$
\mathcal{V}=\{\Gamma\in{\bf k}^4\mid \mathcal{D}(\Gamma)\not=0; \mathcal{C}_1(\Gamma),\mathcal{C}_2(\Gamma)\not\in \mathfrak{T}, \big(\mathcal{C}_1(\Gamma),\mathcal{C}_2(\Gamma),\mathcal{C}_3(\Gamma)\big)\in\mathcal{\tilde V}\}.
$$
For $\Gamma\in\mathcal{V}$, we also define $\mathcal{C}(\Gamma)=\{\mathcal{C}_1(\Gamma),\mathcal{C}_2(\Gamma),\mathcal{C}_3(\Gamma)\}\subset{\bf k}^2$.

Let us consider the action of the cyclic group $C_2$ on ${\bf k}^*\setminus \{1\}$ defined by the equality ${}^{\rho}\alpha=\alpha^{-1}$ for $\alpha\in{\bf k}^*\setminus \{1\}$.
Let us fix some set of representatives of orbits under this action and denote it by ${\bf k_{>1}^*}$. For example, if ${\bf k}=\mathbb{C}$, then one can take $\mathbb{C}_{>1}^*=\{\alpha\in\mathbb{C}^*\mid |\alpha|>1\}\cup\{\alpha\in\mathbb{C}^*\mid |\alpha|=1,0<arg(\alpha)\le \pi\}$. For $(\alpha,\beta,\gamma)\in {\bf k}^2\times{\bf k}^*_{>1},$ we define $$\mathcal{C}(\alpha,\beta,\gamma)=\left\{\big(\alpha\gamma,(1-\alpha)\gamma\big),\left(\frac{\beta}{\gamma},\frac{1-\beta}{\gamma}\right)\right\}\subset{\bf k}^2.$$


Now, from \cite{kv17} we have the result that classifies all $2$-dimensional algebras over an algebraically closed field {\bf k} (we use the corrected version of \cite{kv17}, in which the family ${\bf D}_4(\alpha,\beta)$ that was missing in the published text is added; see Appendix~A).

\begin{theorem} \label{alg}
 Any non-trivial $2$-dimensional ${\bf k}$-algebra can be represented by a unique structure from \hyperlink{tab1}{Table~1}.
\end{theorem}

In this paper, we consider algebraically closed field {\bf k} of characteristic zero.
\par\medskip\noindent\hypertarget{tab1}{\textbf{Table 1.}} Algebraic classification of $2$-dimensional algebras\par\nopagebreak\smallskip
{\small\setlength{\tabcolsep}{4pt}\renewcommand{\arraystretch}{1.25}
\begin{longtable}{|>{\raggedright\arraybackslash}p{0.300\textwidth}|>{\raggedright\arraybackslash}p{0.620\textwidth}|}
\hline
\textbf{Designation} & \textbf{Multiplication table} \\
\hline
\endfirsthead
\hline
\textbf{Designation} & \textbf{Multiplication table} \\
\hline
\endhead
${\bf A}_1(\alpha),\allowbreak\  \alpha\in{\bf k}$ & \mbox{$e_1e_1=e_1+e_2,$} \mbox{$e_1e_2=\alpha e_2,$} \mbox{$e_2e_1= (1-\alpha) e_2,$} \mbox{$e_2e_2=0$} \\
\hline
${\bf A}_2$ & \mbox{$e_1e_1=e_2,$} \mbox{$e_1e_2=e_2,$} \mbox{$e_2e_1= -e_2,$} \mbox{$e_2e_2=0$} \\
\hline
${\bf A}_3$ & \mbox{$e_1e_1=e_2,$} \mbox{$e_1e_2= 0,$} \mbox{$e_2e_1=0,$} \mbox{$e_2e_2=0$} \\
\hline
${\bf A}_4(\alpha),\allowbreak\  \alpha\in{\bf k_{\geq 0}}$ & \mbox{$e_1e_1= \alpha e_1+ e_2,$} \mbox{$e_1e_2= e_1+ \alpha e_2,$} \mbox{$e_2e_1= - e_1,$} \mbox{$e_2e_2=0$} \\
\hline
${\bf B}_1(\alpha),\allowbreak\  \alpha\in{\bf k}$ & \mbox{$e_1e_1=0,$} \mbox{$e_1e_2=(1-\alpha)e_1 + e_2,$} \mbox{$e_2e_1=\alpha e_1 - e_2,$} \mbox{$e_2e_2=0$} \\
\hline
${\bf B}_2(\alpha),\allowbreak\  \alpha\in{\bf k}$ & \mbox{$e_1e_1=0,$} \mbox{$e_1e_2=(1-\alpha)e_1,$} \mbox{$e_2e_1=\alpha e_1,$} \mbox{$e_2e_2=0$} \\
\hline
${\bf B}_3$ & \mbox{$e_1e_1=0,$} \mbox{$e_1e_2= e_2,$} \mbox{$e_2e_1=-e_2,$} \mbox{$e_2e_2=0$} \\
\hline
${\bf C}(\alpha,\beta),\allowbreak\  (\alpha,\beta)\in{\bf k}\times{\bf k_{\geq 0}}$ & \mbox{$e_1e_1=e_2,$} \mbox{$e_1e_2= (1-\alpha)e_1+\beta e_2,$} \mbox{$e_2e_1=\alpha e_1-\beta e_2,$} \mbox{$e_2e_2=e_2$} \\
\hline
${\bf D}_1(\alpha,\beta),\allowbreak\  (\alpha,\beta)\in \mathcal{U}$ & \mbox{$e_1e_1=e_1,$} \mbox{$e_1e_2=(1-\alpha)e_1+ \beta e_2,$} \mbox{$e_2e_1=\alpha e_1 - \beta e_2,$} \mbox{$e_2e_2=0$} \\
\hline
${\bf D}_{2}(\alpha,\beta),\allowbreak\  (\alpha,\beta)\in{\bf k}^2\setminus \mathcal{T}$ & \mbox{$e_1e_1=e_1,$} \mbox{$e_1e_2=\alpha e_2,$} \mbox{$e_2e_1=\beta e_2,$} \mbox{$e_2e_2=0$} \\
\hline
${\bf D}_{3}(\alpha,\beta),\allowbreak\  (\alpha,\beta)\in{\bf k}^2\setminus \mathcal{T}$ & \mbox{$e_1e_1=e_1,$} \mbox{$e_1e_2=e_1+ \alpha e_2,$} \mbox{$e_2e_1=- e_1 + \beta e_2,$} \mbox{$e_2e_2=0$} \\
\hline
${\bf D}_{4}(\alpha,\beta),\allowbreak\  (\alpha,\beta)\in{\bf k}^2$ & \mbox{$e_1e_1=e_1,$} \mbox{$e_1e_2=(1-\alpha)e_1+ \beta e_2,$} \mbox{$e_2e_1=\alpha e_1 + (1-\beta) e_2,$} \mbox{$e_2e_2=0$} \\
\hline
${\bf E}_1(\alpha,\beta,\gamma,\delta),\allowbreak\  (\alpha,\beta,\gamma,\delta)\in \mathcal{V}$ & \mbox{$e_1e_1=e_1,$} \mbox{$e_1e_2=\alpha e_1+ \beta e_2,$} \mbox{$e_2e_1=\gamma e_1 + \delta e_2,$} \mbox{$e_2e_2=e_2$} \\
\hline
${\bf E}_2(\alpha,\beta,\gamma),\allowbreak\  (\alpha,\beta,\gamma)\in {\bf k}^3\setminus {\bf k}\times\mathcal{T}$ & \mbox{$e_1e_1=e_1,$} \mbox{$e_1e_2=(1-\alpha) e_1+ \beta e_2,$} \mbox{$e_2e_1=\alpha e_1 + \gamma e_2,$} \mbox{$e_2e_2=e_2$} \\
\hline
${\bf E}_3(\alpha,\beta,\gamma),\allowbreak\  (\alpha,\beta,\gamma)\in {\bf k}^2\times{\bf k}^*_{>1}$ & \mbox{$e_1e_1=e_1,$} \mbox{$e_1e_2=(1-\alpha)\gamma e_1+ \frac{\beta}{\gamma} e_2,$} \mbox{$e_2e_1=\alpha\gamma e_1 + \frac{1-\beta}{\gamma} e_2,$} \mbox{$e_2e_2=e_2$} \\
\hline
${\bf E}_4$ & \mbox{$e_1e_1=e_1,$} \mbox{$e_1e_2=e_1+ e_2,$} \mbox{$e_2e_1=0,$} \mbox{$e_2e_2=e_2$} \\
\hline
${\bf E}_5(\alpha),\allowbreak\  \alpha\in{\bf k}$ & \mbox{$e_1e_1=e_1,$} \mbox{$e_1e_2=(1-\alpha) e_1+ \alpha e_2,$} \mbox{$e_2e_1=\alpha e_1 + (1-\alpha) e_2,$} \mbox{$e_2e_2=e_2$} \\
\hline
\end{longtable}}

\section{The algebraic classification of $2$-dimensional rigid (quasi-conservative) and conservative algebras}

\subsection{The algebraic classification of $2$-dimensional rigid algebras}\label{sec:rigid}
The following result presents this classification.

\begin{theorem}
Let $\bf A$ be a $2$-dimensional rigid algebra over an algebraically closed field {\bf k} of characteristic zero,
then $\bf A$ is isomorphic to one of the non-isomorphic algebras presented in \hyperlink{tab2}{Table~2}.
\end{theorem}
\par\medskip\noindent\hypertarget{tab2}{\textbf{Table 2.}} Algebraic classification of 2-dimensional rigid algebras\par\nopagebreak\smallskip
{\small\setlength{\tabcolsep}{4pt}\renewcommand{\arraystretch}{1.25}
\begin{longtable}{|>{\raggedright\arraybackslash}p{0.230\textwidth}|>{\raggedright\arraybackslash}p{0.140\textwidth}|>{\raggedright\arraybackslash}p{0.550\textwidth}|}
\hline
$\mathfrak R_{01}(\alpha),\allowbreak\  \alpha\in{\bf k}$ & ${\bf A}_1(\alpha)$ & \mbox{$e_1e_1=e_1+e_2,$} \mbox{$e_1e_2= \alpha e_2,$} \mbox{$e_2e_1= (1-\alpha) e_2,$} \mbox{$e_2e_2=0$} \\
\hline
$\mathfrak R_{02}$ & ${\bf A}_2$ & \mbox{$e_1e_1=e_2,$} \mbox{$e_1e_2=e_2,$} \mbox{$e_2e_1= -e_2,$} \mbox{$e_2e_2=0$} \\
\hline
$\mathfrak R_{03}$ & ${\bf A}_3$ & \mbox{$e_1e_1=e_2,$} \mbox{$e_1e_2= 0,$} \mbox{$e_2e_1=0,$} \mbox{$e_2e_2=0$} \\
\hline
$\mathfrak R_{04}(\alpha),\allowbreak\  \alpha\in{\bf k}$ & ${\bf B}_2(\alpha)$ & \mbox{$e_1e_1=0,$} \mbox{$e_1e_2=(1-\alpha)e_1,$} \mbox{$e_2e_1=\alpha e_1,$} \mbox{$e_2e_2=0$} \\
\hline
$\mathfrak R_{05}$ & ${\bf B}_3$ & \mbox{$e_1e_1=0,$} \mbox{$e_1e_2= e_2,$} \mbox{$e_2e_1=-e_2,$} \mbox{$e_2e_2=0$} \\
\hline
$\mathfrak R_{06}$ & ${\bf C}(1,0)$ & \mbox{$e_1e_1=e_2,$} \mbox{$e_1e_2= 0,$} \mbox{$e_2e_1= e_1,$} \mbox{$e_2e_2=e_2$} \\
\hline
$\mathfrak R_{07}$ & ${\bf D}_1(0,0)$ & \mbox{$e_1e_1=e_1,$} \mbox{$e_1e_2=e_1,$} \mbox{$e_2e_1=0,$} \mbox{$e_2e_2=0$} \\
\hline
$\mathfrak R_{08}$ & ${\bf D}_1(1/2,0)$ & \mbox{$e_1e_1=e_1,$} \mbox{$e_1e_2= \frac{1}{2} e_1,$} \mbox{$e_2e_1= \frac{1}{2} e_1,$} \mbox{$e_2e_2=0$} \\
\hline
$\mathfrak R_{09}$ & ${\bf D}_1(1,1)$ & \mbox{$e_1e_1=e_1,$} \mbox{$e_1e_2= e_2,$} \mbox{$e_2e_1= e_1 - e_2,$} \mbox{$e_2e_2=0$} \\
\hline
${\mathfrak R}_{10}(\alpha, \beta),\allowbreak\  (\alpha,\beta)\in{\bf k}^2\setminus \mathcal{T}$ & ${\bf D}_{2}(\alpha,\beta)$ & \mbox{$e_1e_1=e_1,$} \mbox{$e_1e_2=\alpha e_2,$} \mbox{$e_2e_1=\beta e_2,$} \mbox{$e_2e_2=0$} \\
\hline
$\mathfrak R_{11}$ & ${\bf D}_4(0,1)$ & \mbox{$e_1e_1=e_1,$} \mbox{$e_1e_2=e_1+e_2,$} \mbox{$e_2e_1=0,$} \mbox{$e_2e_2=0$} \\
\hline
$\mathfrak R_{12}$ & ${\bf D}_4(1,1)$ & \mbox{$e_1e_1=e_1,$} \mbox{$e_1e_2=e_2,$} \mbox{$e_2e_1=e_1,$} \mbox{$e_2e_2=0$} \\
\hline
$\mathfrak R_{13}(\alpha),\allowbreak\  (\alpha, 1+\alpha, 1+\alpha,\alpha)\in \mathcal{V}$ & ${\bf E}_1(\alpha, 1+\alpha, 1+\alpha,\alpha)$ & \mbox{$e_1e_1=e_1,$} \mbox{$e_1e_2=\alpha e_1+ (1+\alpha) e_2,$} \mbox{$e_2e_1= (1+\alpha) e_1 + \alpha e_2,$} \mbox{$e_2e_2=e_2$} \\
\hline
$\mathfrak R_{14}(\alpha, \beta),\allowbreak\  (0,\alpha, \beta, 0)\in \mathcal{V}$ & ${\bf E}_1(0,\alpha, \beta, 0)$ & \mbox{$e_1e_1=e_1,$} \mbox{$e_1e_2= \alpha e_2,$} \mbox{$e_2e_1= \beta e_1,$} \mbox{$e_2e_2=e_2$} \\
\hline
$\mathfrak R_{15}(\alpha),\allowbreak\  \alpha\in{\bf k}^*$ & ${\bf E}_2(1, 1, \alpha)$ & \mbox{$e_1e_1=e_1,$} \mbox{$e_1e_2= e_2,$} \mbox{$e_2e_1= e_1 + \alpha e_2,$} \mbox{$e_2e_2=e_2$} \\
\hline
$\mathfrak R_{16}(\alpha),\allowbreak\  \alpha\in{\bf k}\setminus\left\{1\right\}$ & ${\bf E}_2(1, \alpha, 0)$ & \mbox{$e_1e_1=e_1,$} \mbox{$e_1e_2= \alpha e_2,$} \mbox{$e_2e_1= e_1,$} \mbox{$e_2e_2=e_2$} \\
\hline
$\mathfrak R_{17}(\alpha),\allowbreak\  \alpha\in{\bf k}^*_{>1}$ & ${\bf E}_3(1, \alpha, \alpha)$ & \mbox{$e_1e_1=e_1,$} \mbox{$e_1e_2= e_2,$} \mbox{$e_2e_1=\alpha e_1 + \frac{1-\alpha}{\alpha} e_2,$} \mbox{$e_2e_2=e_2$} \\
\hline
$\mathfrak R_{18}(\alpha),\allowbreak\  \alpha\in{\bf k}^*_{>1}$ & ${\bf E}_3(1/\alpha,1, \alpha)$ & \mbox{$e_1e_1=e_1,$} \mbox{$e_1e_2=(\alpha-1) e_1+ \frac{1}{\alpha} e_2,$} \mbox{$e_2e_1= e_1,$} \mbox{$e_2e_2=e_2$} \\
\hline
$\mathfrak R_{19}(\alpha),\allowbreak\  \alpha\in{\bf k}^*_{>1}$ & ${\bf E}_3(1, 1, \alpha)$ & \mbox{$e_1e_1=e_1,$} \mbox{$e_1e_2=\frac{1}{\alpha} e_2,$} \mbox{$e_2e_1= \alpha e_1,$} \mbox{$e_2e_2=e_2$} \\
\hline
$\mathfrak R_{20}$ & ${\bf E}_3(0, 0, -1)$ & \mbox{$e_1e_1=e_1,$} \mbox{$e_1e_2=-e_1,$} \mbox{$e_2e_1=- e_2,$} \mbox{$e_2e_2=e_2$} \\
\hline
$\mathfrak R_{21}$ & ${\bf E}_4$ & \mbox{$e_1e_1=e_1,$} \mbox{$e_1e_2=e_1+ e_2,$} \mbox{$e_2e_1=0,$} \mbox{$e_2e_2=e_2$} \\
\hline
$\mathfrak R_{22}(\alpha),\allowbreak\  \alpha\in{\bf k}$ & ${\bf E}_5(\alpha),\allowbreak\  \alpha\in{\bf k}$ & \mbox{$e_1e_1=e_1,$} \mbox{$e_1e_2=(1-\alpha) e_1+ \alpha e_2,$} \mbox{$e_2e_1=\alpha e_1 + (1-\alpha) e_2,$} \mbox{$e_2e_2=e_2$} \\
\hline
\end{longtable}}

\begin{remark}\label{remarkcases}
Let ${\bf A}=(\mathbb V, \mathcal P)$ be a 2-dimensional algebra, with $\{e_1, e_2\}$ a basis of ${\bf A}$. We can prove that $\bf A$ is rigid by  $\mathcal F$ and $\phi$  for the following cases:
    \begin{align*}
     \textrm{(1.a)} & \quad   a=b=x=y=e_1; & \textrm{(1.b)} & \quad   a=b=x=y=e_2;   \\
     \textrm{(2.a)} & \quad   a=x=y=e_1, b=e_2; & \textrm{(2.b)} & \quad   a=x=y=e_2, b=e_1;  \\
     \textrm{(3.a)} & \quad   b=x=y=e_1, a=e_2;& \textrm{(3.b)} & \quad   b=x=y=e_2, a=e_1;  \\
     \textrm{(4.a)} & \quad   a=b=y=e_1, x=e_2;& \textrm{(4.b)} & \quad   a=b=y=e_2, x=e_1;  \\
     \textrm{(5.a)} & \quad   a=b=x=e_1, y=e_2;& \textrm{(5.b)} & \quad   a=b=x=e_2, y=e_1;  \\
     \textrm{(6.a)} & \quad   x=y=e_1, a=b=e_2;& \textrm{(6.b)} & \quad   x=y=e_2, a=b=e_1;  \\
     \textrm{(7.a)} & \quad   a=x=e_1, b=y=e_2;& \textrm{(7.b)} & \quad   a=x=e_2, b=y=e_1;  \\
     \textrm{(8.a)} & \quad   b=x=e_1, a=y=e_2;& \textrm{(8.b)} & \quad   b=x=e_2, a=y=e_1.
  \end{align*}
\end{remark}

Let $\mathcal F:\mathbb V\times \mathbb V\rightarrow \mathbb V$  be a bilinear map:
$$\begin{array}{ll}
\mathcal F(e_1, e_1)=\lambda_1 e_1+\lambda_2 e_2, &  \mathcal F(e_1, e_2)= \mu_1 e_1+\mu_2 e_2, \\
\mathcal F(e_2, e_1)=\tau_1 e_1+\tau_2 e_2, & \mathcal F(e_2, e_2)=\nu _1 e_1+\nu _2 e_2.
\end{array}$$

Also, let $\phi:\mathbb V\times \mathbb V\rightarrow \textbf{k}$ be a bilinear form:
$$\begin{array}{ll}
\mathcal \phi(e_1, e_1)=\phi_{11} & \mathcal \phi(e_1,e_2)=\phi_{12}  \\
\mathcal \phi(e_2, e_1)=\phi_{21} & \mathcal \phi(e_2,e_2)=\phi_{22}.
\end{array}$$

Using the cases described above, we study each family from \hyperlink{tab1}{Table~1}. This procedure lead us to a system of equations that can be solved without too much difficulty. Therefore, we just give the complete procedure for the first case, the other cases can be obtained in an analogous way.

\subsubsection{Algebra  $\bf A_1(\alpha)$, $\alpha\in{\bf k}$}\label{sec:A1}
From Remark \ref{remarkcases}, we obtain the following necessary conditions for the structural constants of $\mathcal F$ and $\phi$ (we have omitted trivial cases):
  {\fontsize{9pt}{11pt}\selectfont\begin{flalign*}
     \textrm{(1.a)} & \; \lambda_1+\phi_{11}=1 \textrm{ and } (-2+\alpha)(2-\alpha-\lambda_{1})+\phi_{11}=0,  & \textrm{(2.a)} & \; \mu_1+\phi_{12}=0 \textrm{ and } (2-\alpha)\mu_1 + \phi_{12}=0, & \\
          \textrm{(3.a)} & \; \tau_1+\phi_{21}=0 \textrm{ and } (2-\alpha)\tau_1 + \phi_{21}=0,  & \textrm{(4.a)} & \; 1-\alpha=(1-\alpha)(\lambda_1+\phi_{11}), & \\
          \textrm{(4.b)} & \;  \alpha(\nu_1+\phi_{22})=0, & \textrm{(5.a)} & \; \alpha=\alpha(\lambda_1+\phi_{11}), & \\ 
          \textrm{(5.b)} & \;  (1-\alpha)(\nu_1+\phi_{22})=0, & \textrm{(6.a)} & \; \nu_1+\phi_{22}=0 \textrm{ and } (2-\alpha)\nu_1 + \phi_{22}=0, &\\
          \textrm{(7.a)} & \; \alpha(\mu_1+\phi_{12})=0, & \textrm{(7.b)} & \; (1-\alpha)(\tau_1+\phi_{21})=0, &\\
          \textrm{(8.a)} & \; \alpha(\tau_1+\phi_{21})=0,   & \textrm{(8.b)} & \; (1-\alpha)(\mu_1+\phi_{12})=0. &
  \end{flalign*}}

Solving this system of equations, we conclude that $\bf A_1(\alpha)$ is rigid for every $\alpha\in{\bf k}$. Namely, we have the following cases:

\begin{enumerate}
\item $\bf A_1(1)$, where $\mathcal F$ is given by:
$$\begin{array}{ll}
\mathcal F(e_1, e_1)= (1-\phi_{11})e_1+\lambda_2 e_2, &  \mathcal F(e_1, e_2)= -\phi_{12} e_1+\mu_2 e_2, \\
\mathcal F(e_2, e_1)=-\phi_{21} e_1+\tau_2 e_2, & \mathcal F(e_2, e_2)=-\phi_{22} e_1+\nu _2 e_2.
\end{array}$$
and $\phi$ is given by:
$$\begin{array}{ll}
\mathcal \phi(e_1, e_1)=\phi_{11}, & \mathcal \phi(e_1,e_2)=\phi_{12},  \\
\mathcal \phi(e_2, e_1)=\phi_{21}, & \mathcal \phi(e_2,e_2)=\phi_{22}.
\end{array}$$
(that is, $\phi$ is arbitrary).

\item $\bf A_1(\alpha)$, $\alpha\neq1$, where $\mathcal F$ is given by:  
$$\begin{array}{ll}
\mathcal F(e_1, e_1)= (3-\alpha)e_1+\lambda_2 e_2, &  \mathcal F(e_1, e_2)= \mu_2 e_2, \\
\mathcal F(e_2, e_1)=\tau_2 e_2, & \mathcal F(e_2, e_2)=\nu _2 e_2,
\end{array}$$
and $\phi$ is given by:
$$\begin{array}{ll}
\mathcal \phi(e_1, e_1)=\alpha-2, & \mathcal \phi(e_1,e_2)=0,  \\
\mathcal \phi(e_2, e_1)=0, & \mathcal \phi(e_2,e_2)=0.
\end{array}$$
\end{enumerate}

\subsubsection{Algebra $\bf A_2$} The algebra $\bf A_2$ is rigid. The choice of $\mathcal  F$ is given by:   
  $$\begin{array}{ll}
    \mathcal F(e_1, e_1) =-e_1+\lambda_2 e_2,   &   \mathcal F(e_1, e_2) =\mu_2 e_2, \\
    \mathcal F(e_2, e_1) =\tau_2 e_2,           &   \mathcal F(e_2, e_2) =\nu _2 e_2,
\end{array}$$
and $\phi=0$.
\subsubsection{Algebra $\bf A_3$.}
The algebra $\bf A_3$ is a Leibniz algebra, and obviously, is rigid for any multiplication $\mathcal F$ and $\phi=0$.

\subsubsection{Algebra $\bf A_4(\alpha)$, $\alpha\in{\bf k_{\geq 0}}$} 
From the conditions (4.b), (5.b), (6.a) and (6.b), we conclude that $\bf A_4(\alpha)$ is not rigid for any $\alpha\in{\bf k_{\geq 0}}$.

\subsubsection{Algebra $\bf B_1(\alpha)$, $\alpha\in{\bf k}$}
The condition $\textrm{(6.a)}$ shows that $\bf B_1(\alpha)$ is not rigid for any $\alpha\in{\bf k}$.

\subsubsection{Algebra $\bf B_2(\alpha)$, $\alpha\in{\bf k}$}
The algebra $\bf B_2(\alpha)$ is rigid. 
If $\alpha=1$ then $\mathcal F$ is given by an arbitrary map, 
otherwise $\mathcal F$ is given by:  
$$\begin{array}{ll}
\mathcal F(e_1, e_1) =\lambda_2 e_2,  &\mathcal F(e_1, e_2) =-\alpha e_1 + \mu_2 e_2, \\
\mathcal F(e_2, e_1) =\tau_2 e_2,     &\mathcal F(e_2, e_2) =\nu _2 e_2 
\end{array}$$
and $\phi=0$.     
\subsubsection{Algebra $\bf B_3$.}
The algebra $\bf B_3$ is a Lie algebra, and obviously, is rigid for any multiplication $\mathcal F$ and $\phi=0$.

\subsubsection{Algebra  $\bf C(\alpha, \beta)$, $(\alpha, \beta) \in{\bf k}\times {\bf k_{\geq 0}}$}      
The algebra $\bf C(\alpha, \beta)$ is rigid if and only if $(\alpha, \beta)=(1, 0)$, for any $\phi,$ and for $\mathcal F$ given by:  
$$\begin{array}{ll}
\mathcal F(e_1, e_1)=-\phi_{11} e_2, &  \mathcal F(e_1, e_2)= e_1-\phi_{12} e_2, \\
\mathcal F(e_2, e_1)= e_1-\phi_{21} e_2, & \mathcal F(e_2, e_2)=(1-\phi_{22}) e_2.
\end{array}$$

\subsubsection{Algebra $\bf D_1(\alpha, \beta)$, $(\alpha, \beta)\in \mathcal{U}$}\label{sec:D1}

The algebra $\bf D_1(\alpha, \beta)$ is rigid in the following cases:   

\begin{enumerate}
\item $\bf D_1(0, 0).$ $\phi=0$ and $\mathcal F$ is given by:  
$$\begin{array}{ll}
\mathcal F(e_1, e_1)= e_1+\lambda_2 e_2, &  \mathcal F(e_1, e_2)= \mu_2 e_2, \\
\mathcal F(e_2, e_1)=\tau_2 e_2, & \mathcal F(e_2, e_2)=\nu _2 e_2.
\end{array}$$

\item $\bf D_1(1/2, 0).$ $\mathcal F$ is given by:  
$$\begin{array}{ll}
\mathcal F(e_1, e_1)=0, &  \mathcal F(e_1, e_2)= -\frac{1}{2} e_1+ e_2, \\
\mathcal F(e_2, e_1)=0, & \mathcal F(e_2, e_2)=\frac{1}{2} e_2,
\end{array}$$
and $\phi$ is given by:  
$$\begin{array}{ll}
\mathcal \phi(e_1, e_1)=1, & \mathcal \phi(e_1,e_2)=\frac{1}{2},  \\
\mathcal \phi(e_2, e_1)=\frac{1}{2}, & \mathcal \phi(e_2,e_2)=0.
\end{array}$$

\item $\bf{D}_1(1, 1)$. $\phi$ is arbitrary and $\mathcal F$ is given by:  
$$\begin{array}{ll}
\mathcal F(e_1, e_1)=(1-\phi_{11}) e_1, &  \mathcal F(e_1, e_2)= -\phi_{12} e_1+ e_2, \\
\mathcal F(e_2, e_1)= -\phi_{21} e_1+ e_2, & \mathcal F(e_2, e_2)= -\phi_{22} e_1+ e_2.
\end{array}$$
\end{enumerate}

\subsubsection{Algebra      $\bf D_2(\alpha, \beta)$, $(\alpha,\beta)\in{\bf k}^2\setminus \mathcal{T}$}
The algebra $\bf D_2(\alpha, \beta)$ is rigid for any $(\alpha,\beta)\in{\bf k}^2\setminus \mathcal{T}$. 
$\phi$ and $\mathcal F$ are given by:
 $$\begin{array}{ll}
    \mathcal F(e_1, e_1) =(1-\phi_{11})e_1 + \lambda_2 e_2,  &\mathcal F(e_1, e_2) =-\phi_{12}e_1+ \mu_2 e_2, \\
    \mathcal F(e_2, e_1) =-\phi_{21}e_1 +\tau_2 e_2,           &\mathcal F(e_2, e_2) =-\phi_{22} e_1+\nu _2 e_2, 
  \end{array}$$
where $\beta\lambda_2=0$, $\beta=\beta\mu_2$, $(2-\alpha)\beta=\beta \tau_2$ and $\beta\nu_2=0.$
Now we have next cases:
\begin{itemize}
 \item If $\beta=0$ then  $\phi$ is  arbitrary  and $\mathcal{F}$ is given by:  
 $$ \begin{array}{ll}
    \mathcal F(e_1, e_1) =(1-\phi_{11})e_1 + \lambda_2 e_2,  &\mathcal F(e_1, e_2) =-\phi_{12}e_1+ \mu_2 e_2, \\
    \mathcal F(e_2, e_1) =-\phi_{21}e_1 +\tau_2 e_2,           &\mathcal F(e_2, e_2) =-\phi_{22} e_1+\nu _2 e_2. 
  \end{array}$$
 \item If $\beta\neq0$ then $\phi$ is  arbitrary  and $\mathcal{F}$ is given by:  
  $$\begin{array}{ll}
    \mathcal F(e_1, e_1) =(1-\phi_{11})e_1,  &\mathcal F(e_1, e_2) =-\phi_{12}e_1+e_2, \\
    \mathcal F(e_2, e_1) =-\phi_{21}e_1 +(2-\alpha) e_2,           &\mathcal F(e_2, e_2) =-\phi_{22} e_1. 
  \end{array}$$
\end{itemize}

\subsubsection{Algebra     $\bf D_3(\alpha, \beta)$, $(\alpha,\beta)\in{\bf k}^2\setminus \mathcal{T}$}

From condition $\textrm{(6.b)}$, we conclude that $\bf D_3(\alpha, \beta)$ is not rigid.

\subsubsection{Algebra $\bf D_4(\alpha, \beta)$, $(\alpha,\beta)\in{\bf k}^2$}

The algebra $\bf D_4(\alpha, \beta)$ is rigid if and only if $(\alpha,\beta)\in\{(0,1),(1,1)\}$. Indeed, the second component of the identity in the case \textrm{(6.a)} contains neither $\mathcal F$ nor $\phi$ and reads $(1-\beta)^2=0$, whence $\beta=1$. For $\beta=1$ the case \textrm{(3.b)} gives $(1-\alpha)\mu_1=-\alpha(1-\alpha)$ and $(1-\alpha)\mu_1=0$, whence $\alpha(1-\alpha)=0$. For the two remaining algebras the system of equations is solved in the following way.

\noindent$\bf D_4(0,1)$: $\mathcal F$ is given by:
$$\begin{array}{ll}
\mathcal F(e_1, e_1)=2e_1+\lambda_2 e_2, &  \mathcal F(e_1, e_2)= \mu_2 e_2, \\
\mathcal F(e_2, e_1)=\tau_2 e_2, & \mathcal F(e_2, e_2)=\nu_2 e_2,
\end{array}$$
and $\phi$ is given by:
$$\begin{array}{ll}
\phi(e_1, e_1)=-1, & \phi(e_1,e_2)=0,  \\
\phi(e_2, e_1)=0, & \phi(e_2,e_2)=0.
\end{array}$$

\noindent$\bf D_4(1,1)$: $\phi$ is arbitrary and $\mathcal F$ is given by:
$$\begin{array}{ll}
\mathcal F(e_1, e_1)=(1-\phi_{11}) e_1, &  \mathcal F(e_1, e_2)= -\phi_{12} e_1+ e_2, \\
\mathcal F(e_2, e_1)= -\phi_{21} e_1+ e_2, & \mathcal F(e_2, e_2)= -\phi_{22} e_1+ e_2.
\end{array}$$

\subsubsection{Algebra $\bf E_1(\alpha, \beta, \gamma, \delta)$, $(\alpha,\beta,\gamma,\delta)\in \mathcal{V}$} 

The algebra $\bf E_1(\alpha, \beta, \gamma, \delta)$ is rigid in the following cases:

\begin{enumerate}
    \item $\bf E_1(\delta',1+\delta', 1+\delta', \delta'), (\delta',1+\delta', 1+\delta', \delta')\in \mathcal{V}$. 
    $\mathcal F$ is given by:  
    $$\begin{array}{ll}
\mathcal F(e_1, e_1)=(2+\delta) e_1, &  \mathcal F(e_1, e_2)= (1-\delta) e_1+ e_2, \\
\mathcal F(e_2, e_1)= e_1+(1-\delta) e_2, & \mathcal F(e_2, e_2)=(2+\delta) e_2.
\end{array}$$
and $\phi$ is given by:  
$$\begin{array}{ll}
\mathcal \phi(e_1, e_1)=-1-\delta, & \mathcal \phi(e_1,e_2)=-1+\delta+2\delta^2,  \\
\mathcal \phi(e_2, e_1)=-1+\delta+2\delta^2, & \mathcal \phi(e_2,e_2)=-1-\delta.
\end{array}$$

    \item $\bf E_1 (0, \beta', \gamma', 0), (0, \beta', \gamma', 0)\in \mathcal{V}$. $\phi$ is arbitrary and  $\mathcal F$ is given by:   {\scriptsize\setlength{\arraycolsep}{2pt}$$\begin{array}{ll}
\mathcal F(e_1, e_1)=\frac{\beta^2 \gamma + \phi_{11}-\gamma\phi_{11}-1}{-1+\beta\gamma}e_1+\frac{(1-\beta)(\beta+\phi_{11})}{-1+\beta\gamma} e_2, &  \mathcal F(e_1, e_2)= \frac{\gamma(\beta-\phi_{12}-1)+\phi_{12}}{-1+\beta\gamma} e_1+\frac{\beta(\gamma-\phi_{12}-1)+\phi_{12}}{-1+\beta\gamma} e_2, \\
\mathcal F(e_2, e_1)=\frac{\gamma(\beta-\phi_{21}-1)+\phi_{21}}{-1+\beta\gamma} e_1+\frac{\beta(\gamma-\phi_{21}-1)+\phi_{21}}{-1+\beta\gamma} e_2, & \mathcal F(e_2, e_2)=\frac{(1-\gamma)(\gamma+\phi_{22})}{-1+\beta\gamma} e_1+\frac{\beta(\gamma^2-\phi_{22})-1+\phi_{22}}{-1+\beta\gamma} e_2.
\end{array}$$}
\end{enumerate}

\subsubsection{Algebra  $\bf E_2(\alpha, \beta, \gamma)$,  $(\alpha,\beta,\gamma)\in {\bf k}^3\setminus {\bf k}\times\mathcal{T}$}

The algebra $\bf E_2(\alpha, \beta, \gamma)$ is rigid for the following parameters:

\begin{enumerate}
\item $\bf E_2(1, 1, \gamma)$, $(1, 1, \gamma)\in {\bf k}^3\setminus {\bf k}\times\mathcal{T}$. 
$\phi$ is arbitrary and $\mathcal F$ is given by:  
$$\begin{array}{ll}
\mathcal F(e_1, e_1)= (1-\phi_{11}) e_1, &  \mathcal F(e_1, e_2)= -\phi_{12} e_1+ e_2, \\
\mathcal F(e_2, e_1)=-\phi_{21} e_1+ e_2, & \mathcal F(e_2, e_2)= (-1-\phi_{22}) e_1+ 2 e_2.
\end{array}$$

\item $\bf E_2(1, \beta, 0)$,  $(1, \beta, 0)\in {\bf k}^3\setminus {\bf k}\times\mathfrak{T}$.  
$\phi$ is arbitrary and $\mathcal F$ is given by:  
$$\begin{array}{ll}
\mathcal F(e_1, e_1)=(1+\beta) e_1+(-\beta-\phi_{11}) e_2, &  \mathcal F(e_1, e_2)=  e_1-\phi_{12} e_2, \\
\mathcal F(e_2, e_1)= e_1-\phi_{21} e_2, & \mathcal F(e_2, e_2)=(1-\phi_{22}) e_2.
\end{array}$$

\end{enumerate}

\subsubsection{Algebra  $\bf E_3(\alpha, \beta, \gamma)$, $(\alpha,\beta,\gamma)\in {\bf k}^2\times{\bf k}^*_{>1}$}
The algebra $\bf E_3(\alpha, \beta, \gamma)$ is rigid for the following parameters: 

\begin{enumerate}
\item $\bf E_3(1, \gamma, \gamma)$, $\gamma\in{\bf k}^*_{>1}$. 
$\phi$ is arbitrary and $\mathcal F$ is given by:  
$$\begin{array}{ll}
\mathcal F(e_1, e_1)=(1-\phi_{11}) e_1, &  \mathcal F(e_1, e_2)= -\phi_{12} e_1+ e_2, \\
\mathcal F(e_2, e_1)=-\phi_{21} e_1+ e_2, & \mathcal F(e_2, e_2)=(-\gamma-\phi_{22}) e_1+(1+\gamma) e_2.
\end{array}$$

\item $\bf E_3(\dfrac{1}{\gamma}, 1, \gamma)$, $\gamma\in{\bf k}^*_{>1}$. 
$\phi$ is arbitrary and $\mathcal F$ is given by:  
$$\begin{array}{ll}
\mathcal F(e_1, e_1)=\dfrac{1+\gamma}{\gamma} e_1-\dfrac{1+\gamma\phi_{11}}{\gamma} e_2, &  \mathcal F(e_1, e_2)=  e_1 - \phi_{12} e_2, \\
\mathcal F(e_2, e_1)=e_1-\phi_{21} e_2, & \mathcal F(e_2, e_2)=(1-\phi_{22}) e_2.
\end{array}$$

\item  $\bf E_3(1, 1, \gamma)$, $\gamma\in{\bf k}^*_{>1}$. 
$\mathcal F$ is given by:  
$$\begin{array}{ll}
\mathcal F(e_1, e_1)=\dfrac{1+\gamma-\gamma^2\lambda_2}{\gamma} e_1+\lambda_2 e_2, &  \mathcal F(e_1, e_2)= (1+\gamma-\gamma\mu_2) e_1+\mu_2 e_2, \\
\mathcal F(e_2, e_1)=(1+\gamma-\gamma\tau_2) e_1+\tau_2 e_2, & \mathcal F(e_2, e_2)=\gamma(1+\gamma-\nu_2) e_1+\nu _2 e_2,
\end{array}$$
and $\phi$ is given by:  
$$\begin{array}{ll}
\mathcal \phi(e_1, e_1)=-\dfrac{1}{\gamma}, & \mathcal \phi(e_1,e_2)=-1,  \\
\mathcal \phi(e_2, e_1)=-1, & \mathcal \phi(e_2,e_2)=-\gamma.
\end{array}$$

\item  $\bf E_3(0, 0, -1)$. 
$\phi=0$ and $\mathcal F$ is given by:  

$$\begin{array}{ll}
\mathcal F(e_1, e_1)=e_1, &  \mathcal F(e_1, e_2)= 2 e_1+ e_2, \\
\mathcal F(e_2, e_1)= e_1+ 2 e_2, & \mathcal F(e_2, e_2)=e_2.
\end{array}$$

\end{enumerate}

\subsubsection{Algebra $\bf E_4$.}
The algebra $\bf E_4$ is rigid for $\mathcal F$ is given by:  
$$\begin{array}{ll}
\mathcal F(e_1, e_1)=2 e_1, &  \mathcal F(e_1, e_2)=  e_1+ e_2, \\
\mathcal F(e_2, e_1)= e_2, & \mathcal F(e_2, e_2)= e_2,
\end{array}$$
and $\phi$ is given by:  
$$\begin{array}{ll}
\mathcal \phi(e_1, e_1)=-1, & \mathcal \phi(e_1,e_2)=-1,  \\
\mathcal \phi(e_2, e_1)=0, & \mathcal \phi(e_2,e_2)=0.
\end{array}$$

\subsubsection{Algebra $\bf E_5(\alpha)$, $\alpha\in{\bf k}$}
The algebra $\bf E_5(\alpha)$ is rigid for any $\alpha\in{\bf k}$. 
$\phi$ is arbitrary  and $\mathcal F$ is given by:  
  $$\begin{array}{ll}
    \mathcal F(e_1, e_1) =\lambda_1 e_1 + (1-\phi_{11}-\lambda_1 ) e_2,   &\mathcal F(e_1, e_2) = \mu_1 e_1 + (1-\phi_{12}-\mu_1) e_2, \\
    \mathcal F(e_2, e_1) = \tau_1 e_1 + (1-\phi_{21}-\tau_1) e_2,         &\mathcal F(e_2, e_2) = \nu_1 e_1 + (1-\phi_{22}-\nu_1) e_2. 
  \end{array}$$

\subsection{The algebraic classification of $2$-dimensional conservative algebras}
As a corollary of the classification of $2$-dimensional rigid algebras, we have the following result.

\begin{theorem}
Let $\bf A$ be a $2$-dimensional conservative algebra  over an algebraically closed field {\bf k} of characteristic zero, then $\bf A$ is isomorphic to one of the non-isomorphic algebras presented in \hyperlink{tab3}{Table~3}. 
\end{theorem}
 \begin{proof}
In the previous results, choose $\phi=0$ when possible, to obtain the conservative algebras. Among the algebras ${\bf D}_4(\alpha,\beta)$ only ${\bf D}_4(1,1)$ is conservative: for $\phi=0$ we obtain $\mathcal F(e_1,e_1)=e_1$ and $\mathcal F(e_1,e_2)=\mathcal F(e_2,e_1)=\mathcal F(e_2,e_2)=e_2$, while the rigid algebra ${\bf D}_4(0,1)$ is not conservative, because $\phi(e_1,e_1)=-1$ is forced. 
 \end{proof}
\par\medskip\noindent\hypertarget{tab3}{\textbf{Table 3.}} Algebraic classification of 2-dimensional conservative algebras\par\nopagebreak\smallskip
{\small\setlength{\tabcolsep}{4pt}\renewcommand{\arraystretch}{1.25}
\begin{longtable}{|>{\raggedright\arraybackslash}p{0.230\textwidth}|>{\raggedright\arraybackslash}p{0.140\textwidth}|>{\raggedright\arraybackslash}p{0.550\textwidth}|}
\hline
$\mathfrak C_{01}$ & ${\bf A}_1(1)$ & \mbox{$e_1e_1=e_1+e_2,$} \mbox{$e_1e_2= e_2,$} \mbox{$e_2e_1= 0,$} \mbox{$e_2e_2=0$} \\
\hline
$\mathfrak C_{02}$ & ${\bf A}_1(2)$ & \mbox{$e_1e_1=e_1+e_2,$} \mbox{$e_1e_2=2 e_2,$} \mbox{$e_2e_1= -e_2,$} \mbox{$e_2e_2=0$} \\
\hline
$\mathfrak C_{03}$ & ${\bf A}_2$ & \mbox{$e_1e_1=e_2,$} \mbox{$e_1e_2=e_2,$} \mbox{$e_2e_1= -e_2,$} \mbox{$e_2e_2=0$} \\
\hline
$\mathfrak C_{04}$ & ${\bf A}_3$ & \mbox{$e_1e_1=e_2,$} \mbox{$e_1e_2= 0,$} \mbox{$e_2e_1=0,$} \mbox{$e_2e_2=0$} \\
\hline
$\mathfrak C_{05}(\alpha),\allowbreak\  \alpha\in{\bf k}$ & ${\bf B}_2(\alpha)$ & \mbox{$e_1e_1=0,$} \mbox{$e_1e_2=(1-\alpha)e_1,$} \mbox{$e_2e_1=\alpha e_1,$} \mbox{$e_2e_2=0$} \\
\hline
$\mathfrak C_{06}$ & ${\bf B}_3$ & \mbox{$e_1e_1=0,$} \mbox{$e_1e_2= e_2,$} \mbox{$e_2e_1=-e_2,$} \mbox{$e_2e_2=0$} \\
\hline
$\mathfrak C_{07}$ & ${\bf C}(1,0)$ & \mbox{$e_1e_1=e_2,$} \mbox{$e_1e_2= 0,$} \mbox{$e_2e_1= e_1,$} \mbox{$e_2e_2=e_2$} \\
\hline
$\mathfrak C_{08}$ & ${\bf D}_1(0,0)$ & \mbox{$e_1e_1=e_1,$} \mbox{$e_1e_2=e_1,$} \mbox{$e_2e_1=0,$} \mbox{$e_2e_2=0$} \\
\hline
$\mathfrak C_{09}$ & ${\bf D}_1(1,1)$ & \mbox{$e_1e_1=e_1,$} \mbox{$e_1e_2= e_2,$} \mbox{$e_2e_1= e_1 - e_2,$} \mbox{$e_2e_2=0$} \\
\hline
$\mathfrak C_{10}(\alpha, \beta),\allowbreak\  (\alpha, \beta) \in{\bf k}^2 \setminus \mathcal{T}$ & ${\bf D}_{2}(\alpha,\beta)$ & \mbox{$e_1e_1=e_1,$} \mbox{$e_1e_2=\alpha e_2,$} \mbox{$e_2e_1=\beta e_2,$} \mbox{$e_2e_2=0$} \\
\hline
$\mathfrak C_{11}$ & ${\bf D}_4(1,1)$ & \mbox{$e_1e_1=e_1,$} \mbox{$e_1e_2=e_2,$} \mbox{$e_2e_1=e_1,$} \mbox{$e_2e_2=0$} \\
\hline
$\mathfrak C_{12}(\alpha, \beta),\allowbreak\  (0, \alpha, \beta, 0)\in{\mathcal V}$ & ${\bf E}_1(0,\alpha, \beta, 0)$ & \mbox{$e_1e_1=e_1,$} \mbox{$e_1e_2= \alpha e_2,$} \mbox{$e_2e_1= \beta e_1,$} \mbox{$e_2e_2=e_2$} \\
\hline
$\mathfrak C_{13}(\alpha),\allowbreak\  \alpha\in{\bf k}^*$ & ${\bf E}_2(1, 1, \alpha)$ & \mbox{$e_1e_1=e_1,$} \mbox{$e_1e_2= e_2,$} \mbox{$e_2e_1= e_1 + \alpha e_2,$} \mbox{$e_2e_2=e_2$} \\
\hline
$\mathfrak C_{14}(\alpha),\allowbreak\  \alpha\in{\bf k}\setminus \left\{1\right\}$ & ${\bf E}_2(1, \alpha, 0)$ & \mbox{$e_1e_1=e_1,$} \mbox{$e_1e_2= \alpha e_2,$} \mbox{$e_2e_1= e_1,$} \mbox{$e_2e_2=e_2$} \\
\hline
$\mathfrak C_{15}(\alpha),\allowbreak\  \alpha\in{\bf k}^*_{>1}$ & ${\bf E}_3(1, \alpha, \alpha)$ & \mbox{$e_1e_1=e_1,$} \mbox{$e_1e_2= e_2,$} \mbox{$e_2e_1=\alpha e_1 + \frac{1-\alpha}{\alpha} e_2,$} \mbox{$e_2e_2=e_2$} \\
\hline
$\mathfrak C_{16}(\alpha),\allowbreak\  \alpha\in{\bf k}^*_{>1}$ & ${\bf E}_3(1/\alpha,1, \alpha)$ & \mbox{$e_1e_1=e_1,$} \mbox{$e_1e_2=(\alpha-1) e_1+ \frac{1}{\alpha} e_2,$} \mbox{$e_2e_1= e_1,$} \mbox{$e_2e_2=e_2$} \\
\hline
$\mathfrak C_{17}$ & ${\bf E}_3(0,0, -1)$ & \mbox{$e_1e_1=e_1,$} \mbox{$e_1e_2=-e_1,$} \mbox{$e_2e_1=- e_2,$} \mbox{$e_2e_2=e_2$} \\
\hline
$\mathfrak C_{18}(\alpha),\allowbreak\  \alpha\in{\bf k}$ & ${\bf E}_5(\alpha)$ & \mbox{$e_1e_1=e_1,$} \mbox{$e_1e_2=(1-\alpha) e_1+ \alpha e_2,$} \mbox{$e_2e_1=\alpha e_1 + (1-\alpha) e_2,$} \mbox{$e_2e_2=e_2$} \\
\hline
\end{longtable}}

\section{The classification of $2$-dimensional terminal algebras}   

The aim of this section is to present the algebraic and geometric classification of the class of the terminal algebras.

\subsection{The algebraic classification of $2$-dimensional terminal algebras}   
As a corollary of the classification of $2$-dimensional conservative algebras, we have the following result.

\begin{theorem}
Let $\bf A$ be a $2$-dimensional terminal algebra over an algebraically closed field {\bf k} of characteristic zero, 
then $\bf A$ is isomorphic to one of the non-isomorphic algebras presented in \hyperlink{tab4}{Table~4}.
\end{theorem}

 \begin{proof}
In the previous results, choose $\mathcal F(a,b)=2/3{\mathcal P}(a,b)+1/3{\mathcal P}(b,a)$ when possible, to obtain the terminal algebras. The conservative algebra ${\bf D}_4(1,1)$ is not terminal: its multiplication $\mathcal F$ is unique and $\mathcal F(e_1,e_2)=e_2\neq\frac23e_2+\frac13e_1=\frac23{\mathcal P}(e_1,e_2)+\frac13{\mathcal P}(e_2,e_1)$. 
 \end{proof}
\par\medskip\noindent\hypertarget{tab4}{\textbf{Table 4.}} Algebraic classification of 2-dimensional terminal algebras\par\nopagebreak\smallskip
{\small\setlength{\tabcolsep}{4pt}\renewcommand{\arraystretch}{1.25}
\begin{longtable}{|>{\raggedright\arraybackslash}p{0.230\textwidth}|>{\raggedright\arraybackslash}p{0.140\textwidth}|>{\raggedright\arraybackslash}p{0.550\textwidth}|}
\hline
$\mathfrak T_{01}$ & ${\bf A}_1(1)$ & \mbox{$e_1e_1=e_1+e_2,$} \mbox{$e_1e_2= e_2,$} \mbox{$e_2e_1= 0,$} \mbox{$e_2e_2=0$} \\
\hline
$\mathfrak T_{02}$ & ${\bf A}_1(2)$ & \mbox{$e_1e_1=e_1+e_2,$} \mbox{$e_1e_2=2 e_2,$} \mbox{$e_2e_1= -e_2,$} \mbox{$e_2e_2=0$} \\
\hline
$\mathfrak T_{03}$ & ${\bf A}_3$ & \mbox{$e_1e_1=e_2,$} \mbox{$e_1e_2= 0,$} \mbox{$e_2e_1=0,$} \mbox{$e_2e_2=0$} \\
\hline
$\mathfrak T_{04}$ & ${\bf B}_2(1)$ & \mbox{$e_1e_1=0,$} \mbox{$e_1e_2=0,$} \mbox{$e_2e_1= e_1,$} \mbox{$e_2e_2=0$} \\
\hline
$\mathfrak T_{05}$ & ${\bf B}_2(-1)$ & \mbox{$e_1e_1=0,$} \mbox{$e_1e_2=2 e_1,$} \mbox{$e_2e_1=-e_1,$} \mbox{$e_2e_2=0$} \\
\hline
$\mathfrak T_{06}$ & ${\bf B}_3$ & \mbox{$e_1e_1=0,$} \mbox{$e_1e_2= e_2,$} \mbox{$e_2e_1=-e_2,$} \mbox{$e_2e_2=0$} \\
\hline
$\mathfrak T_{07}(\alpha),\allowbreak\  \alpha\in {\bf k}\setminus \left\{1\right\}$ & ${\bf D}_{2}(\alpha,0)$ & \mbox{$e_1e_1=e_1,$} \mbox{$e_1e_2=\alpha e_2,$} \mbox{$e_2e_1=0,$} \mbox{$e_2e_2=0$} \\
\hline
$\mathfrak T_{08}(\alpha),\allowbreak\  \alpha\in {\bf k}\setminus \left\{2\right\}$ & ${\bf D}_{2}(\alpha,3-2\alpha\neq0)$ & \mbox{$e_1e_1=e_1,$} \mbox{$e_1e_2=\alpha e_2,$} \mbox{$e_2e_1=(3-2\alpha) e_2,$} \mbox{$e_2e_2=0$} \\
\hline
$\mathfrak T_{09}$ & ${\bf E}_1(0,0,0,0)$ & \mbox{$e_1e_1=e_1,$} \mbox{$e_1e_2=0,$} \mbox{$e_2e_1=0,$} \mbox{$e_2e_2=e_2$} \\
\hline
$\mathfrak T_{10}(\alpha),\allowbreak\  \alpha\in {\bf k}$ & ${\bf E}_5(\alpha)$ & \mbox{$e_1e_1=e_1,$} \mbox{$e_1e_2=(1-\alpha) e_1+ \alpha e_2,$} \mbox{$e_2e_1=\alpha e_1 + (1-\alpha) e_2,$} \mbox{$e_2e_2=e_2$} \\
\hline
\end{longtable}}

\subsection{Closures of orbit of $2$-dimensional terminal algebras}
\label{graf2}

This subsection is devoted to show the geometric classification of the variety of $2$-dimensional terminal algebras.

Geometric classification is an interesting subject, which was studied in various papers 
(see, for example, \cite{BS99,IKV17,IKV18,GRH,KPPV16}). One of the problems in this direction is to describe all degenerations in a variety of algebras of some fixed dimension satisfying some set of identities. For example, this problem was solved
for $2$-dimensional pre-Lie algebras  in \cite{BB09}, for $2$-dimensional Jordan algebras in \cite{jor2}, 
for $3$-dimensional Novikov algebras in \cite{BB14}, 
for $4$-dimensional Lie algebras in \cite{BS99},  for $4$-dimensional Zinbiel, and nilpotent Leibniz algebras in \cite{KPPV16},
for nilpotent $5$- and $6$-dimensional Lie algebras in \cite{GRH}, \cite{Se90}, 
and for nilpotent $5$-dimensional and $6$-dimensional Malcev algebras in \cite{kpv}. 

We will denote by ${\bf \mathfrak{T}_2}$ the set of all $\mu \in \mathcal{A}_2:= Hom(\mathbb V \otimes \mathbb V,\mathbb V)$ such that $\mu$ is a representation of a terminal algebra. Consider the Zariski topology on $\mathcal{A}_2$, giving  it  the structure of an affine variety. Since ${\bf \mathfrak{T}_2} \subset \mathcal{A}_2$ is defined by a set of polynomial identities, then ${\bf \mathfrak{T}_2}$ is a Zariski-closed of the variety of all two dimensional algebras. Thus, any $\mu\in {\bf \mathfrak{T}_2}$ is determined by the structure constants $c_{ij}^k\in \bf  k$ ($i,j,k=1,2$) such that $\mu(e_i\otimes e_j)=c_{ij}^1e_1+c_{ij}^2e_2$. 

In the previous subsection, we gave a decomposition of ${\bf \mathfrak{T}_2}$ into $GL(\mathbb  V)$-orbits (acting by conjugation), i.e, a classification, up to isomorphism, of the terminal algebras. In this subsection, we will describe the closures of the orbits of $\mu\in{\bf \mathfrak{T}_2}$, denoted by $\overline{O(\mu)}$, and we will give a geometric classification of ${\bf \mathfrak{T}_2}$, 
by  describing it's irreducible components. 

For that purpose, we will introduce some definitions. Let $\bf A$ and $\bf B$ be two $2$-dimensional algebras and let $\mu,\lambda \in {\bf \mathfrak{T}_2}$ represents $\bf A$ and $\bf B$ respectively.
We say that $\bf A$ degenerates to $\bf B$ and write $\bf A\to \bf B$ if $\lambda\in\overline{O(\mu)}$.
Moreover, we have $\overline{O(\lambda)}\subset\overline{O(\mu)}$. Hence, the definition of a degeneration does not depend on the choice of the representative $\mu$ and $\lambda$. We write $\bf A\not\to\bf  B$ if $\lambda\not\in\overline{O(\mu)}$.

Now, let $\bf A(*):=\{\bf A(\alpha)\}_{\alpha\in I}$ be a set of $2$-dimensional algebras and $\mu(\alpha)\in\mathfrak{T}_2$ represent $\bf A(\alpha)$ for $\alpha\in I$.
If $\lambda\in\overline{\{O(\mu(\alpha))\}_{\alpha\in I}}$, then we write $\bf A(*)\to \bf B$ and say that $\bf A(*)$ degenerates to $\bf B$. Again, in the opposite case we write $\bf A(*)\not\to \bf B$.

Let $\bf A(*)$, $\bf B$, $\mu(\alpha)$ ($\alpha\in I$) and $\lambda$ be as above. Let $c_{ij}^k$ ($i,j,k=1,2$) be the structure constants of $\lambda$ in the basis $e_1,e_2$. If we construct $a_i^j:{\bf  k}^*\to {\bf  k}$ ($i,j=1,2$) and $f: {\bf  k}^* \to I$ such that $a_1^1(t)e_1+a_1^2(t)e_2$ and $a_2^1(t)e_1+a_2^2(t)e_2$ is a basis of $\mathbb V$ for any  $t\in{\bf  k}^*$, and the structure constants of $\mu({f(t)})$ in this basis are $c_{ij}^k(t)\in{\bf  k}[t]$ such that $c_{ij}^k(0)=c_{ij}^k$, then $\bf A(*)\to \bf B$. In this case  $(a_1^1(t)e_1+a_1^2(t)e_2,a_2^1(t)e_1+a_2^2(t)e_2)$ and $f(t)$ are called a {\it parametrized basis} and a {\it parametrized index} for $A(*)\to B$ respectively. Note that in the case of $|I|=1$ we only need a parametrized basis.

The following lemma holds:
\begin{lemma}\label{lemmadeg}
Let $\bf A\to \bf B$ be a proper degeneration. Then it follows:
\begin{enumerate}
    \item $dim\,Aut({\bf A})<dim\,Aut({\bf B})$.
    \item $dim\,\left[\bf A, \bf A \right]\geq dim\,\left[\bf B, \bf B\right]$.
\end{enumerate}
\end{lemma}

The following result in \cite{KPPV16} gives us a constructive method to prove non-degenerations.

\begin{lemma}\label{lemmaseparating}
Let $\mathcal{B}$ be a  Borel subgroup of $GL(\mathbb V)$ and let $\mathcal{R}$ be a closed subset of ${\bf \mathfrak{T}_2}$ such that $\mathcal{R}$ is stable under the action of $\mathcal{B}$. If $\bf A(*)\to \bf B$ and $\mu(\alpha)$, a structure representing  of $\bf A(\alpha)$, is in $\mathcal{R}$ for all $\alpha\in I$, then there exists a representation $\lambda$ of $\bf B$ such that $\lambda \in\mathcal{R}$.
\end{lemma}

Constructing a set $\mathcal{R}$ under the conditions of the previous result, such that $\mu(\alpha)\in\mathcal{R}$ for any $\alpha\in I$ and $O(\lambda)\cap \mathcal{R}=\varnothing$, gives us the non-degeneration $\bf A(*)\not\to \bf B$. In this case, we call $\mathcal{R}$ a {\it separating set} for $\bf A(*)\not\to \bf B$. 
In this paper, we always choose $\mathcal{B}$ as the lower triangular matrices. To prove a non-degeneration, we present the separating set and omit any verification, which can be obtained by straightforward calculations.

\begin{theorem} \label{geotermpuntos}
The variety of all $2$-dimensional terminal algebras  has the graph of primary degenerations presented on  \hyperlink{fig1}{Figure~1}.
\end{theorem}
\par\medskip
\noindent\begin{minipage}{\textwidth}\begin{center}
\hypertarget{fig1}{\textbf{Figure 1.}}
Graph of primary degenerations of the variety of $2$-dimensional terminal algebras.
\par\nopagebreak\medskip\nopagebreak
\begin{adjustbox}{max width=\textwidth}\begin{tikzpicture}[->,>=stealth',shorten >=0.05cm,auto,node distance=1cm,
                    thick,main node/.style={rectangle,draw,fill=gray!10,rounded corners=1.5ex,font=\sffamily \scriptsize \bfseries },rigid node/.style={rectangle,draw,fill=black!20,rounded corners=1.5ex,font=\sffamily \scriptsize \bfseries },style={draw,font=\sffamily \scriptsize \bfseries }]

\node (0)   {$0$};
\node (0aa) [below  of=0]      {};

\node (1) [below  of=0aa]      {$1$};
\node (1aa) [below  of=1]      {};

\node (2) [below  of=1aa]      {$2$};

\node (4) [below  of=2]      {$4$};

\node            (00)  [right of =0]   {};                     
\node            (01)  [right of =00]   {};                     
\node            (02)  [right of =01]   {};                     
\node            (03)  [right of =02]   {};        
\node            (04)  [right of =03]   {};    
\node            (05)  [right of =04]   {};    
\node[main node] (T9)  [right of =05]                       {$\mathfrak{T}_{09}$};

\node            (051)  [right of =T9]   {};
\node            (052)  [right of =051]   {};

\node            (12)  [right of =1]   {};                     
        \node[main node] (T1)  [right of =12]                       {$\mathfrak{T}_{01}$};

\node            (A1q)  [right of =T1]   {};     

    \node[main node] (T2)  [right of =A1q]                       {$\mathfrak{T}_{02}$};
    
\node            (A1qq)  [right of =T2]   {};     

	\node[main node] (T7)  [right of =A1qq]                       {$\mathfrak{T}_{07}(\alpha)$ };
\node            (17)  [right of =T7]   {};                     

	\node[main node] (T8)  [right of =17]                       {$\mathfrak{T}_{08}(\alpha)$ };

\node            (20)  [right of =T8]   {};        

	\node[main node] (T4)  [right of =20]                       {$\mathfrak{T}_{04}$ };
	
\node            (21)  [right of =T4]   {};     

\node[main node] (T5)  [right of =21]                       {$\mathfrak{T}_{05}$ };

\node            (22)  [right of =T5]   {};                     

\node            (25)  [right of =2]   {};                     
\node            (25b)  [right of =25]   {};    
	\node[main node] (T10)  [right of =25b]                       {$\mathfrak{T}_{10}(\alpha)$};

\node            (26)  [right of =T10]   {};                     
\node            (26c)  [right of =26]   {};   
\node            (26cc)  [right of =26c]   {};   
	\node[main node] (T3)  [right of =26cc]                       {$\mathfrak{T}_{3}$};
	
\node            (26)  [right of =T3]   {};                     
\node            (26b)  [right of =26]   {};    
\node            (26bb)  [right of =26b]   {};    
	\node[main node] (T6)  [right of =26bb]                       {$\mathfrak{T}_{6}$};
	
\node (q1)[right of=T6]{};
\node (q2)[right of=q1]{};
\node (q21)[right of=q2]{};
\node (q31)[right of=q21]{};

\node (q3)[right of=q31]   {};

\node (q4)[below  of=q3]   {};

\node (1160)[right of=4]{};
\node (1161)[right of=1160]{};
\node (1162)[right of=1161]{};
\node (1163)[right of=1162]{};
\node (1164)[right of=1163]{};
\node (1165)[right of=1164]{};
	\node[main node] (K2) [ right  of=1165]       {${\bf k}^2$};
 
\path[every node/.style={font=\sffamily\small}]

   (T9)   edge node[above=0, right=-15, fill=white]{\tiny $\alpha=0$ } (T7) 
   (T9)   edge node[above=0, right=-15, fill=white]{\tiny $\alpha=1$ } (T8) 
   
   (T1)   edge  (T3) 
   (T2)   edge  (T3) 
   (T7)   edge  (T3) 
   (T8)   edge  (T3) 
   (T4)   edge  (T3) 
   (T5)   edge  (T3)   
   (T1)   edge node[above=0, right=-15, fill=white]{\tiny $\alpha=1$ } (T10) 
   (T2)   edge node[above=0, right=-15, fill=white]{\tiny $\alpha=2$ } (T10) 
   (T10)   edge  (K2) 
   (T3)   edge  (K2) 
   (T6)   edge  (K2)     ;

\end{tikzpicture}\end{adjustbox}

\end{center}\end{minipage}

\begin{proof} 
    This proof is mostly based in results in \cite{kv17} and in the ideas described there to construct separating sets. All primary degenerations are showed in \hyperlink{tab5}{Table~5} and all non-degenerations that do not follow from Lemma \ref{lemmadeg} are presented in \hyperlink{tab6}{Table~6}.
\par\medskip\noindent\hypertarget{tab5}{\textbf{Table 5.}} Degenerations required to prove Theorem \ref{geotermpuntos}\par\nopagebreak\smallskip
{\small\setlength{\tabcolsep}{4pt}\renewcommand{\arraystretch}{1.25}
\begin{longtable}{|>{\raggedright\arraybackslash}p{0.300\textwidth}|>{\raggedright\arraybackslash}p{0.580\textwidth}|}
\hline
\textbf{Degeneration} & \textbf{Parametrized basis} \\
\hline
\endfirsthead
\hline
\textbf{Degeneration} & \textbf{Parametrized basis} \\
\hline
\endhead
${\mathfrak{T}_{01}} \to {\mathfrak{T}_{03}}$ & \mbox{$E_1^t=te_1,$} \mbox{$E_2^t=t^2e_2$} \\
\hline
${\mathfrak{T}_{01}} \to {\mathfrak{T}_{10}}(1)$ & \mbox{$E_1^t=e_1,$} \mbox{$E_2^t=e_1+t^{-1}e_2$} \\
\hline
${\mathfrak{T}_{02}} \to {\mathfrak{T}_{03}}$ & \mbox{$E_1^t=te_1,$} \mbox{$E_2^t=t^2e_2$} \\
\hline
${\mathfrak{T}_{02}} \to {\mathfrak{T}_{10}}(2)$ & \mbox{$E_1^t=e_1,$} \mbox{$E_2^t=e_1+t^{-1}e_2$} \\
\hline
${\mathfrak{T}_{04}} \to {\mathfrak{T}_{03}}$ & \mbox{$E_1^t=e_1+te_2,$} \mbox{$E_2^t=te_1$} \\
\hline
${\mathfrak{T}_{05}} \to {\mathfrak{T}_{03}}$ & \mbox{$E_1^t=e_1+te_2,$} \mbox{$E_2^t=te_1$} \\
\hline
${\mathfrak{T}_{07}}(\alpha) \to {\mathfrak{T}_{03}}$ & \mbox{$E_1^t=te_1+te_2,$} \mbox{$E_2^t=t^2e_1+\alpha t^2e_2$} \\
\hline
${\mathfrak{T}_{08}}(\alpha) \to {\mathfrak{T}_{03}}$ & \mbox{$E_1^t=te_1+te_2,$} \mbox{$E_2^t=t^2e_1+(3-\alpha)t^2e_2$} \\
\hline
${\mathfrak{T}_{09}} \to {\mathfrak{T}_{07}}(0)$ & \mbox{$E_1^t=e_1,$} \mbox{$E_2^t=te_2$} \\
\hline
${\mathfrak{T}_{09}} \to {\mathfrak{T}_{08}}(1)$ & \mbox{$E_1^t=e_1+e_2,$} \mbox{$E_2^t=te_2$} \\
\hline
\end{longtable}}
\par\medskip\noindent\hypertarget{tab6}{\textbf{Table 6.}} Non-degenerations required to prove Theorem \ref{geotermpuntos}\par\nopagebreak\smallskip
{\small\setlength{\tabcolsep}{4pt}\renewcommand{\arraystretch}{1.25}
\begin{longtable}{|>{\raggedright\arraybackslash}p{0.380\textwidth}|>{\raggedright\arraybackslash}p{0.520\textwidth}|}
\hline
\textbf{Non-degeneration} & \textbf{Separating set} \\
\hline
\endfirsthead
\hline
\textbf{Non-degeneration} & \textbf{Separating set} \\
\hline
\endhead
${\mathfrak{T}_{01}} \not\to {\mathfrak{T}_{06}},\allowbreak\  {\mathfrak{T}_{10}}(\alpha\neq1)$ & $\mathcal R=\{\mu\mid$ \mbox{$c_{22}^1=0,$} \mbox{$c_{12}^2=c_{11}^1,$} \mbox{$c_{21}^1=c_{22}^2,$} \mbox{$c_{12}^1=0,$} \mbox{$c_{21}^2=0\}$} \\
\hline
${\mathfrak{T}_{02}} \not\to {\mathfrak{T}_{06}},\allowbreak\  {\mathfrak{T}_{10}}(\alpha\neq2)$ & $\mathcal R=\{\mu\mid$ \mbox{$c_{22}^1=0,$} \mbox{$c_{11}^1=\dfrac{c_{12}^2}{2},$} \mbox{$c_{21}^1=2c_{22}^2,$} \mbox{$c_{12}^1=-c_{22}^2,$} \mbox{$c_{21}^2=- \dfrac{c_{12}^2}{2}\}$} \\
\hline
${\mathfrak{T}_{04}} \not\to {\mathfrak{T}_{06}}$ & $\mathcal R=\{\mu\mid$ \mbox{$c_{22}^1=c_{22}^2=c_{12}^1=0,$} \mbox{$2 c_{12}^2+c_{21}^2=- c_{11}^1,$} \mbox{$c_{11}^1 (c_{12}^2+c_{21}^2)= c_{11}^2 c_{21}^1\}$} \\
\hline
${\mathfrak{T}_{05}} \not\to {\mathfrak{T}_{06}}$ & $\mathcal R=\{\mu\mid$ \mbox{$c_{22}^1=c_{22}^2=0,$} \mbox{$c_{12}^1=-2 c_{21}^1,$} \mbox{$c_{21}^2=c_{11}^1,$} \mbox{$c_{11}^1(c_{11}^1+c_{12}^2)+c_{11}^2 c_{21}^1=0\}$} \\
\hline
${\mathfrak{T}_{07}}(\alpha) \not\to {\mathfrak{T}_{06}},\allowbreak\  {\mathfrak{T}_{10}}(\gamma)$ & $\mathcal R=\{\mu\mid$ \mbox{$c_{12}^1=0,$} \mbox{$c_{12}^2= \alpha c_{11}^1,$} \mbox{$c_{21}^1=0,$} \mbox{$c_{21}^2=0,$} \mbox{$c_{22}^1=0,$} \mbox{$c_{22}^2=0\}$} \\
\hline
${\mathfrak{T}_{08}}(\alpha) \not\to {\mathfrak{T}_{06}},\allowbreak\  {\mathfrak{T}_{10}}(\gamma)$ & $\mathcal R=\{\mu\mid$ \mbox{$c_{12}^1=0,$} \mbox{$c_{12}^2=\alpha c_{11}^1,$} \mbox{$c_{21}^1=0,$} \mbox{$c_{21}^2= (3-2\alpha)c_{11}^1,$} \mbox{$c_{22}^1=0,$} \mbox{$c_{22}^2=0\}$} \\
\hline
${\mathfrak{T}_{09}} \not\to {\mathfrak{T}_{04}},\allowbreak\  {\mathfrak{T}_{05}},\allowbreak\  {\mathfrak{T}_{06}},\allowbreak\  {\mathfrak{T}_{10}}(\alpha),\allowbreak\  {\mathfrak{T}_{07}}(\alpha\neq 0),\allowbreak\  {\mathfrak{T}_{08}}(\alpha\neq 1)$ & $\mathcal R=\{\mu\mid$ \mbox{$c_{22}^1=0,$} \mbox{$c_{21}^1=0,$} \mbox{$c_{12}^1=0,$} \mbox{$c_{12}^2=c_{21}^2,$} \mbox{$c_{11}^2 c_{22}^2= -c_{21}^2 (c_{11}^1-c_{21}^2)\}$} \\
\hline
\end{longtable}}

\end{proof}

The following result gives us the description of the closure of the orbits of the infinite series in ${\bf \mathfrak{T}_2}$. For a parametric series of algebras $X$, we will denote by $ X(*)$, the set of all algebras $X(\Gamma)$ that are defined and are terminal, i.e, $\mathfrak T_{07}(*)=\left\{ \mathfrak T_{07}(\alpha): \alpha\in {\bf k}\setminus \left\{1\right\}\right\}$, $\mathfrak T_{08}(*)=\left\{ \mathfrak T_{08}(\alpha): \alpha\in {\bf k}\setminus \left\{2\right\}\right\}$ and $\mathfrak T_{10}(*)=\left\{ \mathfrak T_{10}(\alpha): \alpha\in {\bf  k}\right\}$.

\begin{theorem} \label{geotermseries} 

A description of the closures of the orbits of the infinite series of algebras in ${\bf \mathfrak{T}_2}$ is given in the following table.
$$\begin{array}{|l|l|}
\hline
\mbox{Set}   &  \mbox{Description}  \\
\hline
\hline

\overline{O({\mathfrak{T}_{07}}(*))}\Tstrut\Bstrut  & {\mathfrak{T}_{07}}(*), {\mathfrak{T}_{03}}, {\mathfrak{T}_{01}}, 
{\mathfrak{T}_{04}}, {\mathfrak{T}_{10}}(1), {\bf  k}^2   \\
\hline

\overline{O({\mathfrak{T}_{08}}(*))}\Tstrut\Bstrut  & {\mathfrak{T}_{08}}(*), {\mathfrak{T}_{03}}, {\mathfrak{T}_{02}}, {\mathfrak{T}_{05}}, {\mathfrak{T}_{10}}(2), {\mathfrak{T}_{07}}(3/2), {\bf  k}^2   \\
\hline

\overline{O({\mathfrak{T}_{10}}(*))}\Tstrut\Bstrut  & {\mathfrak{T}_{10}}(*), {\mathfrak{T}_{06}}, {\bf  k}^2  \\
\hline

\end{array}$$

\end{theorem}
\begin{proof} 

Degenerations are proved using the parametrized bases and indexes in \hyperlink{tab7}{Table~7}. Also, the non-degenerations can be proven using the separating sets in \hyperlink{tab8}{Table~8}.

\par\medskip\noindent\hypertarget{tab7}{\textbf{Table 7.}} Degenerations required to prove Theorem \ref{geotermseries}\par\nopagebreak\smallskip
{\small\setlength{\tabcolsep}{4pt}\renewcommand{\arraystretch}{1.25}
\begin{longtable}{|>{\raggedright\arraybackslash}p{0.270\textwidth}|>{\raggedright\arraybackslash}p{0.400\textwidth}|>{\raggedright\arraybackslash}p{0.210\textwidth}|}
\hline
\textbf{Degeneration} & \textbf{Parametrized basis} & \textbf{Parametrized index} \\
\hline
\endfirsthead
\hline
\textbf{Degeneration} & \textbf{Parametrized basis} & \textbf{Parametrized index} \\
\hline
\endhead
${\mathfrak{T}_{07}}(*) \to {\mathfrak{T}_{01}}$ & \mbox{$E_1^t= e_1+e_2,$} \mbox{$E_2^t=t e_2$} & $f(t)= 1+t$ \\
\hline
${\mathfrak{T}_{07}}(*) \to {\mathfrak{T}_{04}}$ & \mbox{$E_1^t= e_2,$} \mbox{$E_2^t=t e_1$} & $f(t)= 1+t^{-1}$ \\
\hline
${\mathfrak{T}_{08}}(*) \to {\mathfrak{T}_{02}}$ & \mbox{$E_1^t= e_1+e_2,$} \mbox{$E_2^t=t e_2$} & $f(t)= 2-t$ \\
\hline
${\mathfrak{T}_{08}}(*) \to {\mathfrak{T}_{05}}$ & \mbox{$E_1^t= e_2,$} \mbox{$E_2^t=t e_1$} & $f(t)= 2-t^{-1}$ \\
\hline
${\mathfrak{T}_{10}}(*) \to {\mathfrak{T}_{06}}$ & \mbox{$E_1^t=t e_1,$} \mbox{$E_2^t=e_1 - e_2$} & $f(t)= t^{-1}$ \\
\hline
\end{longtable}}
\par\medskip\noindent\hypertarget{tab8}{\textbf{Table 8.}} Non-degenerations required to prove Theorem \ref{geotermseries}\par\nopagebreak\smallskip
{\small\setlength{\tabcolsep}{4pt}\renewcommand{\arraystretch}{1.25}
\begin{longtable}{|>{\raggedright\arraybackslash}p{0.380\textwidth}|>{\raggedright\arraybackslash}p{0.520\textwidth}|}
\hline
\textbf{Non-degeneration} & \textbf{Separating set} \\
\hline
\endfirsthead
\hline
\textbf{Non-degeneration} & \textbf{Separating set} \\
\hline
\endhead
${\mathfrak{T}_{07}}(*) \not\to {\mathfrak{T}_{02}},\allowbreak\  {\mathfrak{T}_{05}},\allowbreak\  {\mathfrak{T}_{08}}(\alpha),\allowbreak\  {\mathfrak{T}_{06}},\allowbreak\  {\mathfrak{T}_{10}}(\alpha\neq 1)$ & $\mathcal R=\{\mu\mid$ \mbox{$c_{12}^1=0,$} \mbox{$c_{21}^1=0,$} \mbox{$c_{21}^2=0,$} \mbox{$c_{22}^1=0,$} \mbox{$c_{22}^2=0\}$} \\
\hline
${\mathfrak{T}_{08}}(*) \not\to {\mathfrak{T}_{01}},\allowbreak\  {\mathfrak{T}_{04}},\allowbreak\  {\mathfrak{T}_{07}}(\alpha\neq 3/2),\allowbreak\  {\mathfrak{T}_{06}},\allowbreak\  {\mathfrak{T}_{10}}(\alpha\neq 2)$ & $\mathcal R=\{\mu\mid$ \mbox{$c_{12}^1=0,$} \mbox{$c_{21}^1=0,$} \mbox{$c_{21}^2=3 c_{11}^1-2 c_{12}^2,$} \mbox{$c_{22}^1=0,$} \mbox{$c_{22}^2=0\}$} \\
\hline
${\mathfrak{T}_{10}}(*) \not\to {\mathfrak{T}_{03}}$ & $\mathcal R=\{\mu\mid$ \mbox{$c_{22}^1=0,$} \mbox{$c_{11}^2=0,$} \mbox{$c_{22}^2=c_{12}^1+c_{21}^1,$} \mbox{$c_{11}^1=c_{12}^2+c_{21}^2\}$} \\
\hline
\end{longtable}}

\end{proof}

From Theorem \ref{geotermpuntos} and Theorem \ref{geotermseries}  we obtain the following corollaries.

\begin{corollary}

The lattice of subsets for ${\bf \mathfrak{T}_2}$ is presented on
\hyperlink{fig2}{Figure~2}. 
\end{corollary}
\par\medskip
\noindent\begin{minipage}{\textwidth}\begin{center}
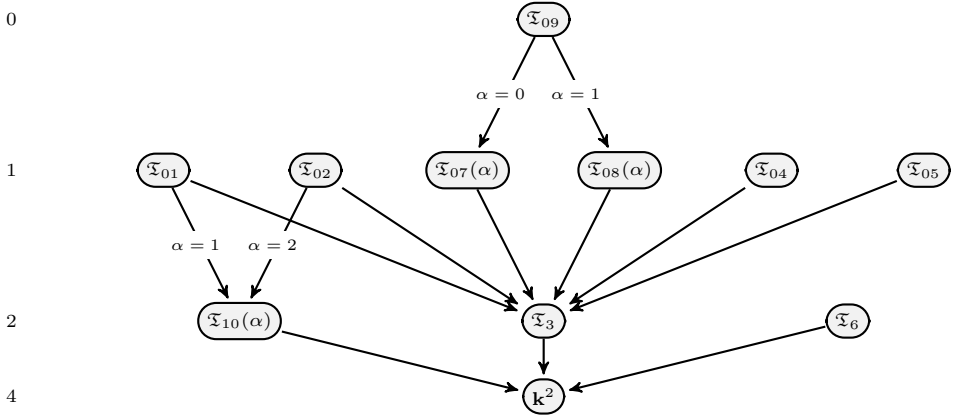
\hypertarget{fig2}{\textbf{Figure 2.}}  Lattice of subsets of the variety of $2$-dimensional terminal algebras.
\par\nopagebreak\medskip\nopagebreak
{\tiny
\begin{adjustbox}{max width=\textwidth}\begin{tikzpicture}[-,draw=gray!50,node distance=0.93cm,
                   ultra thick,main node/.style={rectangle, fill=gray!50,font=\sffamily \scriptsize \bfseries },style={draw,font=\sffamily \scriptsize \bfseries }]

\node (4) {$4$};
\node (4r) [right  of=4] {};
\node (4rr) [right  of=4r] {};
\node (4rrr) [right  of=4rr] {};
\node (3) [right  of=4rrr]      {$3$};
\node (3r) [right  of=3] {};
\node (3rr) [right  of=3r] {};
\node (3rrr) [right  of=3rr] {};
\node (3rrrr) [right  of=3rrr] {};
\node (2) [right  of=3rrrr]      {$2$};
\node (2r) [right  of=2] {};
\node (2rr) [right  of=2r] {};
\node (2rrr) [right  of=2rr] {};
\node (0) [right  of=2rrr]      {$0$};

	\node (C5)  [below of =4]     	{};
	\node[main node] (T7A)  [below of =C5]     	{$\overline{O\big({\mathfrak{T}_{07}}(*)\big)}$};
	\node (T7AR)  [below of =T7A]     	{};
	\node[main node] (T9)  [below of =T7AR]     	{$\overline{O\big({\mathfrak{T}_{09}}\big)}$};
	\node (T8AR)  [below of =T9]     	{};
	\node[main node] (T8A)  [below of =T8AR]     	{$\overline{O\big({\mathfrak{T}_{08}}(*)\big)}$};

	\node[main node] (T70)  [below of =3]                       {$\overline{O\big({\mathfrak{T}_{07}}(0)\big)}$};
	\node[main node] (T1)  [below of =T70]                       {$\overline{O\big({\mathfrak{T}_{01}}\big)}$};
	\node[main node] (T4)  [below of =T1]                       {$\overline{O\big({\mathfrak{T}_{04}}\big)}$};

	\node[main node] (T732)  [below of =T4]                       {$\overline{O\big({\mathfrak{T}_{07}}(3/2)\big)}$};

	\node[main node] (T5)  [below of =T732]                       {$\overline{O\big({\mathfrak{T}_{05}}\big)}$};
	\node[main node] (T2)  [below of =T5]                       {$\overline{O\big({\mathfrak{T}_{02}}\big)}$};
	\node[main node] (T81)  [below of =T2]                       {$\overline{O\big({\mathfrak{T}_{08}}(1)\big)}$};
	\node[main node] (T10A)  [below of =T81]                       {$\overline{O\big({\mathfrak{T}_{10}}(*)\big)}$};

	\node (T101R)  [below of =2] {};
	\node[main node] (T101)  [below of =T101R]                       {$\overline{O\big({\mathfrak{T}_{10}}(1)\big)}$};
	\node (T3R)  [below of =T101]   {};
	\node[main node] (T3)  [below of =T3R]                       {$\overline{O\big({\mathfrak{T}_{03}}\big)}$};	
	\node (T102R)  [below of =T3]{};	
	\node[main node] (T102)  [below of =T102R]                       {$\overline{O\big({\mathfrak{T}_{10}}(2)\big)}$};	
	\node (T6R)  [below of =T102]  {};	
	\node[main node] (T6)  [below of =T6R]                       {$\overline{O\big({\mathfrak{T}_{06}}\big)}$};

\node (0b) [below of =0] {};
\node (0bb) [below of =0b] {};
\node (0bbb) [below of =0bb] {};
\node (0bbbb) [below of =0bbb] {};

	\node[main node] (K2) [ below  of =0bbbb]       {${\bf k}^2$};



\path 


  (T7A) edge   (T70)
  (T9) edge   (T70) 
  
  (T7A) edge   (T732)  
  (T7A) edge   (T4)  
  (T7A) edge   (T1)

  (T8A) edge   (T732)  
  (T8A) edge   (T5)  
  (T8A) edge   (T2)
 
  (T8A) edge   (T81)
  (T9) edge   (T81)

  (T10A) edge   (T101)
  (T1) edge   (T101) 

   (T70) edge   (T3)
  (T1) edge   (T3)
  (T4) edge   (T3)
  (T732) edge   (T3)
  (T5) edge   (T3)
  (T2) edge   (T3)
  (T81) edge   (T3) 
  
  (T10A) edge   (T102)
  (T2) edge   (T102)
  
  (T10A) edge   (T6)

 (T101) edge   (K2)
 (T3) edge   (K2)
 (T102) edge   (K2)
 (T6) edge   (K2)

 ;
\end{tikzpicture}\end{adjustbox}}
\end{center}\end{minipage}

\begin{corollary}

The irreducible components in the variety of 2-dimensional terminal algebras ${\bf \mathfrak{T}_2}$ are:
$$
\begin{aligned}
\overline{O\big({\mathfrak{T}_{07}}(*)\big)}&=\{{\mathfrak{T}_{07}}(*), {\mathfrak{T}_{03}}, {\mathfrak{T}_{01}}, 
{\mathfrak{T}_{04}}, {\mathfrak{T}_{10}}(1), {\bf  k}^2\},\\
\overline{O\big({\mathfrak{T}_{09}}\big)}&=\{\mathfrak{T}_{09}, {\mathfrak{T}_{07}}(0), {\mathfrak{T}_{08}}(1),{\mathfrak{T}_{03}}, {\bf  k}^2\},\\
\overline{O\big({\mathfrak{T}_{08}}(*)}\big)&=\{{\mathfrak{T}_{08}}(*), {\mathfrak{T}_{03}}, {\mathfrak{T}_{02}}, {\mathfrak{T}_{05}}, {\mathfrak{T}_{10}}(2), {\mathfrak{T}_{07}}(3/2), {\bf  k}^2\},\\
\overline{O\big({\mathfrak{T}_{10}}(*)\big)}&=\{{\mathfrak{T}_{10}}(*), {\mathfrak{T}_{06}}, {\bf  k}^2\}.
\end{aligned}
$$
\end{corollary}

\begin{corollary}

There is only one algebra with open orbit in the variety of $2$-dimensional terminal algebras. It is $\mathfrak{T}_{09}.$
\end{corollary}

\subsection{Some conjectures}
Using the geometric classification of $2$-dimensional terminal algebras we can give two conjectures about the variety of $n$-dimensional terminal algebras.
Let us consider $n$-dimensional analogues of the algebras $\mathfrak{T}_{09}$ and $\mathfrak{T}_{10}(\alpha):$
\begin{enumerate}
    \item[I.] the algebra $ \oplus {\bf  k} e_i  $ defined by
$$ \oplus {\bf  k} e_i  = \langle e_1,  \ldots, e_{n} \rangle \ : \ e_i^2=e_i, \ e_ie_j =0, \ (i\neq j);   $$

    \item[II.] the family $\nu_n(\alpha)$ of algebras defined by
\begin{multline*} \nu_n(\alpha) = \langle e, n_1, \ldots, n_{n-1}\rangle \ : \\ e^2=e, \ en_i=\alpha n_i, \ n_ie=(1-\alpha) n_i  \ (i=1, \ldots, n-1; \alpha \in \bf  k).\end{multline*}
\end{enumerate}

It is easy to see that  algebras $ \oplus {\bf  k} e_i $ and $ \nu_n(\alpha)$ are terminal.

\begin{conjecture}
The $n$-dimensional terminal algebra $\bigoplus {\bf  k} e_i$ has an open orbit.
\end{conjecture}

\begin{conjecture}
$\overline{O\big(\nu_n(\alpha)\big)}$ is an irreducible component of the variety of $n$-dimensional terminal algebras.
\end{conjecture}

\section*{Appendix A. Differences from the published version}
\addcontentsline{toc}{section}{Appendix A. Differences from the published version}

\noindent\textbf{Theorem~\ref{alg} and Table~1:} the classification of $2$-dimensional algebras from \cite{kv17} that was used in the published version of this paper is incomplete: the family ${\bf D}_4(\alpha,\beta)$, $(\alpha,\beta)\in{\bf k}^2$, given by $e_1e_1=e_1$, $e_1e_2=(1-\alpha)e_1+\beta e_2$, $e_2e_1=\alpha e_1+(1-\beta)e_2$, $e_2e_2=0$, was missing there. This family is added to Table~1 (we use the corrected version of \cite{kv17}), and the classifications below are completed accordingly.

\noindent\textbf{Subsection~\ref{sec:rigid} and Table~2:} the new family is considered in a new paragraph. It contains exactly two rigid algebras: ${\bf D}_4(0,1)$  and ${\bf D}_4(1,1)$; they are added to Table~2 as $\mathfrak R_{11}$ and $\mathfrak R_{12}$ and some algebras are renumbered.    

\noindent\textbf{Subsection~\ref{sec:A1}, the theorem of Subsection~\ref{sec:rigid} and Table~2:} the algebra $\mathbf A_1(\alpha)$ is rigid for \emph{every} $\alpha\in\C$, and not only for $\alpha\in\{1,2\}$. It is added and some algebras are renumbered.    

\noindent\textbf{Subsection of conservative algebras and Table~3:} the algebra ${\bf D}_4(1,1)$ is conservative; it is added and some algebras are  renumbered.

\end{document}